\documentclass[a4paper,12pt,oneside]{amsart}
\usepackage[cm]{fullpage}

\usepackage{typearea}
\typearea{15}

\usepackage{mymacros}
\usepackage{caption}
\usepackage[colorlinks=false]{hyperref} % ----- for arXiv

\usepackage{mathtools}
\usepackage[all]{xy}
\usepackage{multirow}

\newcommand\SU{\operatorname{SU}}
\newcommand\Spin{\operatorname{Spin}}
\newcommand\Aff{\operatorname{Aff}}
\newcommand\Stab{\operatorname{Stab}}

\newcommand\AdS{\operatorname{AdS}}

\title{Homogeneous pseudo-Riemannian structures of metrics of Kaluza-Klein type on the three-dimensional anti-de Sitter spacetime}

\author{Fumihiro Ueno}
\email{ueno-fumihiro-dx@alumni.osaka-u.ac.jp}
\date{}

\begin{document}
	
	\begin{abstract}
		We classify homogeneous pseudo-Riemannian structures of a three-parameter family of metrics called Kaluza-Klein type on the three-dimensional anti-de Sitter spacetime $(\AdS_3), $ which is diffeomorphic to $ \bS^1 \times \bR^2, $ with their induced groups of isometries and reductive decompositions.
		%Some of these metrics on $ \bS^1\times\bR^2 $ are regarded as deformations of the standard Lorentzian metric on $ \AdS_3 $ along the timelike and spacelike Killing vector fields, considering analogy of the Hopf fiberation on the three-dimensional sphere.
		%The lightlike analogy of the Hopf fibration on $ \AdS_3 $ is also known; nevertheless, no deformed metrics of Kaluza-Klein type on $ \AdS_3$ along lightlike Hopf fibers are derived from the standard Lorentzian metric. 
		We also obtain the classification of homogeneous almost contact and paracontact metric structures of metrics of Kaluza-Klein type on $ \AdS_3 $ with their isometry groups and reductive decompositions. 
		%In particular, the standard Lorentzian metric on $\AdS_3$ has both Sasakian and paraSasakian structures. 
		%The former associates homogeneous almost contact metric structures, while the latter associates homogeneous almost paracontact metric structures.
		%The former induces homogeneous almost contact metric structures associated with the timelike Hopf fibration, while the latter induces homogeneous almost paracontact metric structures associated with the spacelike Hopf fibration. 
	\end{abstract}
		\maketitle
		\tableofcontents
	\section{Introduction}
	This research aims to investigate homogeneous pseudo-Riemannian structures, which are $ (1,2) $-type tensor fields corresponding to coset spaces, %$ G/H $
	of metrics called Kaluza-Klein type on the three-dimensional anti-de Sitter spacetime $ (\AdS_3) $ summarized in \pref{tb:metric-coset}. 
	First, we give some preliminaries of homogeneous pseudo-Riemannian structures and metrics of Kaluza-Klein type on $ \AdS_3 $.
	Second, we classify homogeneous pseudo-Riemannian structures of the metrics on $ \AdS_3 $ with the induced reductive decompositions and the isometry groups.
	Finally, we obtain a classification of homogeneous almost contact and paracontact metric structures of the metrics on $ \AdS_3 $ with the isometry groups and the reductive decompositions as summarized in \pref{tb:metric-HS-IG-HACM-HAPCM}.
	
	The following lists all the homogeneous pseudo-Riemannian structures obtained for all metrics $ g_{\lambda\mu\nu} = \diag(\lambda,\mu,\nu) $ of Kaluza-Klein type with respect to a left-invariant basis of Pauli-like matrices $ \{X_\alpha\}_{\alpha=0}^2 $ and the dual basis $ \{\theta^\alpha\}_{\alpha=0}^2 $ on $ \AdS_3 $. With the exception of $ S_0, $ each homogeneous pseudo-Riemannian structure is a one-parameter family parametrized by $ t \in \bR $, up to isomorphism with the induced isometry groups in \pref{tb:metric-coset}. 
	\begin{enumerate}\label{en:HS}
		\item $ S_0  = -(\lambda-\mu +\nu)\theta^2\otimes\theta^0\wedge\theta^1 + (\lambda + \mu + \nu)\theta^0\otimes\theta^1\wedge\theta^2 -(\lambda + \mu -\nu) \theta^1\otimes\theta^2\wedge\theta^0, $
		\item $ S_{\text{vol}}(t) =  t(\theta^0\wedge\theta^1\wedge\theta^2), (t \ge 0, t\neq-\lambda=\mu=\nu) $
		\item $ S_\lambda(t) = -\lambda\theta^2\otimes\theta^0\wedge\theta^1+t\theta^0\otimes\theta^1\wedge\theta^2  -\lambda \theta^1\otimes\theta^2\wedge\theta^0, (t\neq \lambda+2\mu, \mu=\nu) $
		\item $ S_\mu(t) = t\theta^2\otimes\theta^0\wedge\theta^1+\mu\theta^0\otimes\theta^1\wedge\theta^2 + \mu \theta^1\otimes\theta^2\wedge\theta^0, (t\neq 2\nu - \mu, -\lambda=\nu) $
		\item $ S_{\text{null}}^{\mp}(t) = \mu\theta^2\otimes \theta^0\wedge\theta^1 + (\mu \pm t)\theta^0\otimes \theta^1\wedge\theta^2 + t\theta^1\otimes\theta^1\wedge\theta^2  -t \theta^0\otimes\theta^2\wedge\theta^0 + (\mu \mp t) \theta^1\otimes \theta^2\wedge\theta^0. (t> 0, -\lambda=\mu=\nu) $
	\end{enumerate}
	
	\begin{table}[ht]	
		\centering
		\phantomsection
		\caption{The coset space representations and corresponding homogeneous pseudo-Riemannian structures (HS) in the above lists for each metric of Kaluza-Klein type on $\AdS_3$ }
		\label{tb:metric-coset}
		\begin{tabular}{|c|c|c|}
			\hline
			metric & coset space & HS \\
			\hline\hline
			\multirow{5}{*}{$-\lambda=\mu=\nu $} & $ \lb\SU(1,1)_{\text{L}}\times\SU(1,1)_{\text{R}}\rb/\SU(1,1)_{\diag}$ & $ S_{\text{vol}}(t) $ \\
			\cline{2-3}
			& $ \lb \SU(1,1)\times U(1) \rb/U(1) $ & $ S_{\lambda}(t) $ \\
			\cline{2-3}
			& $ \lb \SU(1,1)\times \bR \rb/\bR $ & $ S_{\mu}(t) $ \\
			\cline{2-3}
			& $ \lb \SU(1,1)\times \bR \rb/\bR $ & $ S_{\text{null}}^{\mp}(t) $  \\
			\cline{2-3}
			& $ (\SU(1,1)\times\{e\})/\{e\} $ & $ S_0 $ \\
			\hline
			\multirow{2}{*}{$-\lambda\neq\mu=\nu$} & $ \lb \SU(1,1)\times U(1) \rb /U(1)$ & $ S_{\lambda}(t) $  \\
			\cline{2-3}
			& $ (\SU(1,1)\times\{e\})/\{e\} $ & $ S_0 $ \\
			\hline
			\multirow{2}{*}{$-\lambda=\nu\neq\mu$} & $ \lb \SU(1,1)\times \bR\rb /\bR$ & $ S_{\mu}(t)(\simeq S_\nu(t)) $ \\
			\cline{2-3}
			& $ (\SU(1,1)\times\{e\})/\{e\} $ & $ S_0 $ \\
			\hline
			\multirow{2}{*}{$-\lambda=\mu\neq\nu$} & $ \lb \SU(1,1)\times \bR \rb/\bR$ & $ S_{\nu}(t)(\simeq S_\mu(t)) $ \\
			\cline{2-3}
			& $ (\SU(1,1)\times\{e\})/\{e\} $ & $ S_0 $ \\
			\hline
			$-\lambda\neq\mu,\mu\neq\nu,\nu\neq-\lambda$  & $ (\SU(1,1)\times\{e\})/\{e\} $ & $ S_0 $ \\
			\hline
		\end{tabular}
	\end{table}
	
		The following results are obtained from our investigation of \pref{tb:metric-coset}:
	\begin{itemize}
		\item Classification of homogeneous pseudo-Riemannian structures on
		\(\SU(1,1)\cong\AdS_3\), which is the three-dimensional Lorentzian symmetric space, in \pref{th:homogeneous-str-symmetric}.
		\item Determination of all isometry groups and induced reductive decompositions in \pref{sc:HS-classification} for each metric $g_{\lambda\mu\nu}$ of Kaluza-Klein type.
		\item Classifications of the induced homogeneous almost contact and paracontact metric structures given in \pref{sc:HACM} and \pref{sc:HApCM}.
	\end{itemize}
	
	The Berger sphere is known as a homogeneous Riemannian manifold diffeomorphic to the three-dimensional sphere ($ \bS^3 $). A metric on the Berger sphere is obtained by deforming the metric of the standard sphere along the fiber of the Hopf fibration $ \bS^1 \hookrightarrow \bS^3 \to \bS^2. $
	Homogeneous Riemannian structures on Berger spheres were studied in \cite{gadea2005homogeneous}.
	Just as the Berger metrics on $ \bS^3 $ appear as metrics of Kaluza-Klein type studied in \cite{calvaruso2013geometry}, one can similarly obtain Berger-like metrics to deform the standard metric on $ \AdS_3 $ as metrics of Kaluza-Klein type studied in \cite{calvaruso2014metrics}.
	
	There are many similarities between the geometry of $ \AdS_3 $ and $ \bS^3. $
	One such similarity is that both are maximally symmetric spaces with a six-dimensional space of Killing vector fields. 
	As for the differences, since $ \bS^3 $ is an isotropic Riemannian manifold, any choice of Killing vector field generated by the left (resp. right) action yields the same spacelike $ \bS^1 $-fibration known as the Hopf fibration. In contrast, $ \AdS_3 $ is an (anisotropic) Lorentzian manifold that admits both timelike $ \bS^1 $- and spacelike $ \bR $-fibrations obtained by the same construction.  
	These Hopf-like fibrations appear in the deformation of metrics that are homothetic to the standard metric on $ \AdS_3 $ if exactly two elements of $ \{-\lambda, \mu, \nu \} $ in a metric of Kaluza-Klein type $ g_{\lambda\mu\nu} $ coincide. If neither of them coincide, it implies that the stabilizer of the isometry group is trivial.
	%In this case, the metric loses its symmetry of the isometry group.
	We refer to these as the Hopf fibrations on $ \AdS_3, $ as explained in \pref{sc:Hopf-fibration}.
	
	Owning to Cartan's work, a locally symmetric Riemannian manifold is characterized as a parallel curvature tensor field of the Levi-Civita connection $ \nabla $.
	It was extended to the local homogeneity of a Riemannian manifold using a $ (1,2) $-type tensor field $ S $, which defines the canonical (Ambrose-Singer) connection $ \nabla - S $ satisfying the parallel conditions, called a homogeneous Riemannian structure in \cite{tricerri1983homogeneous}.
	In other words, a homogeneous Riemannian structure is a $ (1,2) $-type tensor field preserved by the action of the isometry group.
	In \cite{kirichenko1980homogeneous}, Kiri\v{c}henko generalized the results of Ambrose-Singer \cite{ambrose1958homogeneous} to the case when a homogeneous Riemannian manifold is equipped with geometric structures. 
	A homogeneous almost (para)contact metric structure, which consists of a homogeneous Riemannian structure and an almost (para)contact metric structure preserved by the action of the isometry group, is one of them.
	All these results are extended to the reductive homogeneous pseudo-Riemannian case of an arbitrary signature and summarized in \cite{calvaruso2019pseudo}.
	All homogeneous Riemannian manifolds are reductive, while there are homogeneous pseudo-Riemannian manifolds that are not reductive.
	
	Three-dimensional connected, simply connected, and complete homogeneous Lorentzian manifolds are classified by Calvaruso in \cite{calvaruso2007homogeneous} as well as the Riemannian cases in \cite{sekigawa1977some}. In particular, while the three-dimensional homogeneous Riemannian structures have been classified in \cite{sekigawa1977some}, \cite{abe1986classification}, \cite{calvino2023homogeneous}, \cite{ohno2023homogeneous}, and \cite{inoguchi2025homogeneous}, the three-dimensional homogeneous Lorentzian structures remain unclassified.
	Recently, homogeneous Lorentzian structures on non-symmetric three-dimensional Lie groups have been classified in \cite{Calvaruso:2025aa}. In this article, we classify homogeneous pseudo-Riemannian structures on metrics of Kaluza-Klein type on $\AdS_3 $, which can be regarded as diagonal left-invariant metrics on $ \SU(1,1), $  including both symmetric and non-symmetric cases.
	Although \cite{calvino2023homogeneous} and \cite{Calvaruso:2025aa} includes non-symmetric results of classification of homogeneous psuedo-Riemannian structures on Lie groups, we not only recover some of those but also determine their reductive decompositions and show that the corresponding isometry groups act transitively and (almost) effectively. 
	
	Homogeneous almost contact metric structures on the Berger spheres were also studied in \cite{gadea2005homogeneous}. 
	In this research, we obtain the list of homogeneous almost contact and paracontact metric structures of metrics of Kaluza-Klein type on $ \AdS_3. $
	There is a one-parameter family of homogeneous almost contact (resp. paracontact) metric structures induced by the isometry groups and its reductive decompositions on the standard $ \AdS_3, $ which is a Lorentzian Sasakian (resp. paraSasakian) manifold. If we deform the metric along a timelike (resp. spacelike) Hopf-fiber $ \bS^1$ (resp. $\bR), $ the obtained homogeneous almost contact (resp. paracontact) metric structures are defined on the $ \alpha $-Sasakian (resp. $\beta$-paraSasakian) manifolds.
	In \pref{tb:metric-HS-IG-HACM-HAPCM}, we obtain the list of homogeneous almost contact and paracontact metric structures, up to isomorphism, with the isometry groups that possess non-trivial stabilizers, using the notation in \pref{sc:HACM} and \pref{sc:HApCM}. %Each of the homogeneous almost contact and paracontact metric structure is defined on a pseudo-Riemannian $ \alpha $-Sasakian manifold for the former, and on a Lorentzian $\beta$-paraSasakian manifold for the latter.
	Homogeneous almost (para)contact metric structures and mixed metric 3-structures of Kaluza-Klein type metrics on $\AdS_3$ studied by Calvaruso and Perrone in \cite{calvaruso2014metrics} are summarized in \pref{sc:mixed metric 3-structure} with the trivial isometry groups and reductive decompositions.

	\begin{table}[ht]	
		\centering
		\phantomsection
		\caption{The classification of homogeneous almost (para)contact metric structures (HA(p)CM) with $ S\neq S_0 $ for each type of the Kaluza-Klein type metric on $\AdS_3$ with non-trivial isotropic subgroups of isometries, using notation in \pref{sc:HACM} and \pref{sc:HApCM}}
		\label{tb:metric-HS-IG-HACM-HAPCM}
		\begin{tabular}{|c|c|c|c|}
			\hline
			metric & isometry group & HACM & HApCM   \\
			\hline\hline
			\multirow{2}{*}{$-\lambda=\mu=\nu $} & $\SU(1,1)\times U(1)$ & $(S_\lambda(t),\phi_0,\xi_0,\eta_0)$ & None   \\
			\cline{2-4}
			& $\SU(1,1)\times \bR $ & None & $(S_\mu(t),\phitilde_1,\xi_1,\eta_1)$  \\
			\hline
			$-\lambda\neq\mu=\nu$ & $\SU(1,1)\times U(1)$ & $ (S_\lambda(t),\phi_0,\xi_0,\eta_0) $ & None  \\
			\hline
			$-\lambda=\nu\neq\mu$	& $\SU(1,1)\times \bR $ & None & $ (S_\mu(t),\phitilde_1,\xi_1,\eta_1) $ \\
			\hline
			%$-\lambda=\mu\neq\nu$ & $\SU(1,1)\times \bR$ & None & $ (S_\nu(t),\phitilde_2,\xi_2,\eta_2) $  \\
			%\hline
		\end{tabular}
	\end{table}

	\section{Preliminaries}
	\subsection{Homogeneous pseudo-Riemannian structures}
	Let $ (M,g) $ be a connected pseudo-Riemannian manifold. Let $ \nabla $ be the Levi-Civita connection of $ (M,g), $ and $ R $ be the curvature tensor field that we use the following conventions:
	\begin{equation}
		R(X,Y)Z = \nabla_X\nabla_YZ-\nabla_Y\nabla_XZ-\nabla_{[X,Y]}Z,\quad R(X,Y,Z,W) = g(R(X,Y)Z,W),
	\end{equation}
	for all vector fields $ X,Y,Z,W \in \Gamma(TM). $
	\begin{definition}\label{df:homogenous pseudo-Riemannian manifold}
		A pseudo-Riemannian manifold $ (M,g) $ is called homogeneous if there exists a Lie group of isometries that acts transitively on $ M $.
	\end{definition}
	\begin{definition}\label{df:reductive}
		Let $ G $ be a Lie group, and $ H \subset G $ be a Lie subgroup of $ G. $
		A homogeneous space $ G/H $ is reductive if the Lie algebra $ \frakg $ of $ G $ can be decomposed as $ \frakg = \frakh\oplus \frakm, $ where $ \frakh $ is the Lie algebra of $ H, $ and $ \frakm $ is an $ \Ad|_{H} $-invariant subspace, that is, $ \Ad_h(\frakm)\subset \frakm $ for an arbitrary $ h \in H. $
	\end{definition} 
	\begin{definition}\label{df:homogeneous-structure}
		A homogeneous pseudo-Riemannian structure on $ (M,g) $ is a $ (1,2) $-type tensor field $ S $ such that the canonical connection $ \widetilde\nabla = \nabla -S $ satisfies the following equations:
		\begin{equation}\label{eq:Ambrose-Singer connection}
			\widetilde\nabla g = 0, \quad \widetilde\nabla R = 0, \quad \widetilde\nabla S = 0.
		\end{equation}
		We also denote by $ S $ the associated tensor field of type $ (0,3) $ on $ (M,g) $ defined by 
		 \begin{equation}\label{eq:homogeneous structure tensor}
		 	S(X,Y,Z) = g(S_XY, Z).(X,Y,Z\in \Gamma(TM)) 
		 \end{equation}
		
	\end{definition}
	In particular, Ambrose and Singer gave a characterization of homogeneity in \cite{ambrose1958homogeneous} the Riemannian case, and later generalized in \cite{gadea1997reductive} to the metrics of an arbitrary signature. 
	The characterization is given as follows. A connected, simply connected, and complete reductive homogeneous pseudo-Riemannian manifold is a pseudo-Riemannian manifold admitting a linear connection $ \widetilde\nabla = \nabla - S$ satisfying \eqref{eq:Ambrose-Singer connection}. 
	Let $ (M,g) $ be a reductive homogeneous pseudo-Riemannian manifold, that is $ M=G/H $, where $ G $ is a connected Lie subgroup of isometry of $ (M,g), $ and $ H $ is the isotropy group at a point $ o \in M. $
	The $ \Ad|_{H} $-invariant subspace $ \frakm \subset \frakg $ of the reductive decomposition $ \frakg=\frakh \oplus\frakm $ is identified with $ T_oM $ through the isomorphism
	\begin{equation}
		\tau \colon \frakm \ni X \mapsto X^*|_o \in T_oM, 
	\end{equation}
	where $ X^* $ is the Killing vector field generated by the one-parameter subgroup $ \{\exp(sX)\} $ of $ G $ acting on $ M $ from the left.
	Then the canonical connection $ \widetilde\nabla $ of the reductive homogeneous pseudo-Riemannian manifold $ (M=G/H,g) $ is determined by its value at $ o, $ by 
	\begin{equation}\label{eq:nabla-Lie}
		\widetilde\nabla_{X^*}Y^*|_o = -[X,Y]^*_{\frakm}|_o, \quad X,Y \in \frakm,
	\end{equation} 
	and $ S=\nabla-\widetilde\nabla $ is the homogeneous pseudo-Riemannian structure induced by the reductive decomposition $ \frakg =\frakh \oplus \frakm. $
	
	Conversely, assume that a $ (1,2) $-type tensor field $ S $ on $ (M,g) $, which is a connected, simply connected, and complete pseudo-Riemannian manifold, satisfies the condition \eqref{eq:Ambrose-Singer connection} with respect to the canonical connection $ \widetilde\nabla=\nabla-S. $
	Fix a point $ o\in M, $ and set $ \frakm = T_oM. $
	If $ \Rtilde $ denotes the curvature tensor of the connection $ \widetilde\nabla, $ then the holonomy algebra $ \frakh $ of $ \widetilde\nabla $ is generated by $\Rtilde(X,Y) $ for all $ X,Y \in \frakm. $
	Now we define a Lie algebra structure on the direct sum $ \frakg = \frakh \oplus \frakm $ by
	\begin{align}
		[U,V] & = UV-VU, \\
		[U,X] & = U(X),\\
		[X,Y] & = S_XY-S_YX -\Rtilde(X,Y),
	\end{align}
	for all $ X,Y \in \frakm, $ and $ U,V \in \frakh. $
	Let $ \Gtilde $ be a connected and simply connected Lie group whose Lie algebra is $ \frakg, $ and $ \Htilde $ be a connected Lie subgroup of $ \Gtilde $ whose Lie algebra is $ \frakh. $ Then $ \Gtilde $ acts transitively and almost effectively as a group of isometries on $ (M,g), $ hence $ M $ is diffeomorphic to $ \Gtilde/\Htilde. $
	We can choose the discrete normal subgroup $ \Gamma $ of $ \Gtilde $ that acts trivially on $ M $, then $ G=\Gtilde/\Gamma $ acts transitively and effectively on $ M $ as a group of isometries with the isotropy group $ H=\Htilde/\Gamma. $
	Therefore, $ (M,g) $ is a reductive homogeneous pseudo-Riemannian manifold diffeomorphic to $ G/H. $
	
	\begin{definition}\label{df:isom-HS}
		Let $ (M,g) $ be a reductive homogeneous pseudo-Riemannian manifold, and let $ S $ and $ S' $ be two homogeneous pseudo-Riemannian structures on $ (M,g). $ 
		The two homogeneous structures are said to be isomorphic if there exists an isometry $ \varphi $ on $ (M,g) $
		such that $ \varphi_*(S_X Y) = S'_{\varphi_*X}{\varphi_* Y} $ for $ X,Y \in \Gamma(TM) $ or $ \varphi^*S' = S $ as a $ (0,3) $-type tensor field.
	\end{definition}   
	
	\begin{theorem}\label{th:isom-HS}\cite{tricerri1983homogeneous}
		Let $ (M,g) $ be a reductive homogeneous pseudo-Riemannian manifold, and let $ G $ and $ G' $ be connected Lie subgroups of its isometry group that act transitively on $ M. $
		 Assume that the Lie algebra $\frakg $(resp. $ \frakg' $) of $ G $(resp. $ G' $) has a reductive decompostion $ \frakg = \frakm\oplus \frakh $ (resp.  $ \frakg' = \frakm'\oplus\frakh' $).
		 The homogeneous structures $ S $ and $ S' $ defined from these isometric actions are isomorphic if and only if there exists a Lie algebra isomorphism $ \psi $ from $ \frakg $ to
		 $ \frakg' $ satisfying the following conditions:
		 \begin{enumerate}
		 	\item $\psi(\frakm)=\frakm' $ and $ \psi(\frakh) = \frakh', $
		 	\item $ \psi|_\frakm $ induces an isometry from $ (T_oM, g_o) $ to $ (T_{o'}M, g_{o'}) $ via the maps $ \tau \colon \frakm \ni X \mapsto X^*|_o \in T_oM, $ and $ \tau' \colon \frakm' \ni X' \mapsto X'^*|_{o'} \in T_{o'}M. $
		 \end{enumerate}
	\end{theorem}

	%%---------New subsection
	
	\subsection{Pseudo-Riemannian g-natural metrics on unit sphere bundles}
	Let $ (M,g) $ be an $ m $-dimensional Riemannian manifold.
	We shall see the Ehresmann connection on the tangent bundle as in \cite{sakai1996riemannian}.
	Let $(\scU, x) $ be a local coordinate neighborhood of $ p\in M, $ and $ (T\scU, (x, \xi))$ be a local coordinate neighborhood of $ u \in T_pM $.
	There are two vector bundle structures on $ TM $ with the following local expressions:
	\begin{align}
		\pi_{TTM} \colon T_uTM \ni \left.X^i\dfrac{\partial}{\partial x^i}\right|_p + \left.U^i\frac{\partial}{\partial \xi^i}\right|_u & \mapsto \left.\xi(u)^i\frac{\partial}{\partial x^i}\right|_p =u \in T_pM,\label{eq:TTM-to-TM as pro} \\
		d\pi_{TM} \colon T_uTM \ni \left.X^i\dfrac{\partial}{\partial x^i}\right|_p + \left.U^i\frac{\partial}{\partial \xi^i}\right|_u & \mapsto \left.X^i \frac{\partial}{\partial x^i}\right|_p \in T_pM, \label{eq:dpi}
	\end{align}
	where $ X , U \in \bR^m. $
	\begin{definition}\label{df:Ehresmann-connection}
		The Ehresmann connection of the tangent bundle $ TM $ is the following direct sum decomposition:
		\begin{equation}
			TTM=\scH \oplus \scV,
		\end{equation} 
		where $ \scV  $ is the kernel of $ d\pi_{TM}\colon TTM \to TM. $
		We call $ \scV($resp.$\scH) $ a vertical (resp.horizontal) subbundle  of $ TTM $.
	\end{definition}
	Since $ T_pM \subset TM $ is an $ m $-dimensional submanifold and $ T_uT_pM \subset T_uTM $ is an $ m $-dimensional subspace, $ \scV_u = \ker(d\pi_{TM}|_u \colon T_uTM \to TM) = T_uT_pM. $
	The vertical space is isomorphic to the tangent space of the base manifold at the same point via the following map.
	\begin{equation}\label{eq:vertical-identity}
		\iota_u \colon \scV_u=T_uT_pM \ni U^i\frac{\partial}{\partial \xi^i} \mapsto U^i\frac{\partial}{\partial x^i} \in T_pM.
	\end{equation}
	Now we shall define the connection map $K \colon TTM \to TM $ that is equivalent to the horizontal subbundle of $ TTM $.
	
	\begin{definition}\label{df:connection-map}
		We call $K \colon TTM \to TM$ a connection map if it satisfies the following conditions:
		\begin{enumerate}
			\item The map $ K $ is a bundle homomorphism with two bundle map structures, both between 
			$ TTM $ and $ TM $, as described by \eqref{eq:TTM-to-TM as pro} and \eqref{eq:dpi}, rendering the following diagrams commutative:
			\def\objectstyle{\scriptstyle}
			\def\labelstyle{\scriptscriptstyle}
			\xymatrix{
				TTM \ar[r]^{K} \ar[d]^{\pi_{TTM}} & TM \ar[d]^{\pi_{TM}} \\
				TM \ar[r]^{\pi_{TM}} & M &
			}\label{dg:a}
			\def\objectstyle{\scriptstyle}
			\def\labelstyle{\scriptscriptstyle}
			\xymatrix{
				TTM \ar[r]^{K} \ar[d]^{d\pi_{TM}} & TM \ar[d]^{\pi_{TM}} \\
				TM \ar[r]^{\pi_{TM}} & M &
			}\label{dg:b}
			
			\item The restriction of $ K $ to the vertical subspace is equal to the  isomorphism \eqref{eq:vertical-identity}, i.e.,  \[ K|_{T_uT_pM} = \iota_u \colon T_uT_pM \to T_pM. \]
		\end{enumerate}
		
		The horizontal subbundle is defined for each $u \in T_pM$ by $\scH_u = \ker(K_u \colon T_uTM \to T_pM)$. In this case, the connection $\nabla$ on $TM$ is given by
		\begin{equation}\label{eq:K-nabla}
			K(dX|_u) = \nabla_uX|_p.
		\end{equation}
	\end{definition}
	If a connection $ \nabla $ with the connection form $ \{\omega^i_j \} $ of the local frame $ \{\frac{\partial}{\partial x^i}\}_{i=1}^{m} $ is defined,  the connection map satisfying \eqref{eq:K-nabla} is locally given by
	\begin{equation}
		K (\wtilde) = \lb b^i + \omega^i_{jk}(p)a^j\xi^k(u) \rb\left.\frac{\partial}{\partial x^i}\right|_p. \lb\wtilde = a^i\left.\frac{\partial}{\partial x^i}\right|_p + b^j\left.\frac{\partial}{\partial \xi^j}\right|_u \in T_uT\scU\rb
	\end{equation} 
	It implies that the horizontal distribution $ \{\scH_u\} $ defined by the connection $ \nabla $ is locally given by
	\begin{equation}\label{eq:horizontal-distribution}
		\scH_u = \lc  X^i\left.\frac{\partial}{\partial x^i}\right|_p - \omega^i_{jk}(p) X^j\xi^k(u) \left.\frac{\partial}{\partial \xi^i }\right|_u \in T_uTM \mid X\in \bR^m \rc.	
	\end{equation}
	
	\begin{definition}\label{df:TTM-lifting}
		Let $ u \in T_pM $ be a fixed tangent vector at $ p \in M. $ If we have a connection defined by a connection map $ K, $ the following horizontal and vertical liftings are obtained.
		\begin{enumerate}
			\item  The horizontal lift $ w^h \in T_uTM $ of $ w\in T_pM $ with respect to $ u \in T_pM $ satisfies
			\begin{equation}
				w^h = w -\iota_u^{-1}K(w).
			\end{equation} 
			\item  The vertical lift $ w^v \in T_uTM $ of $ w\in T_pM $ with respect to $ u \in T_pM $ satisfies
			\begin{equation}
				w^v = \iota_u^{-1} K(w).
			\end{equation} 
		\end{enumerate}
	\end{definition}
	Here, we identify the tangent vector $ w\in T_pM $ with a vector in $ T_uTM. $
	
	$ g $-natural metrics form a wide family of metrics on $ TM, $ which depends on six smooth functions from $ \bR^+ $ to $ \bR, $ has been completely described in \cite{abbassi2005some}.
	We apply the notion of a $ g $-natural metric to the unit sphere bundle $ T^1M = \lc (p,u) \in TM \mid g(u,u) = 1 \rc $, by restricting the $ g $-natural metric on $ TM, $ as described in \cite{abbassi2010naturality}. 
	The tangent space of the unit sphere bundle $ T^1M $ is given by
	\begin{equation}
		T_{(p,u) }(T^1M) = \lc w^h + w'^v \mid w \in T_pM, w' \in \{u\}^\perp \subset T_pM \rc.
	\end{equation} 
	Then a $ g $-natural metric $ G $ on $ T^1M $ is defined as follows:
	\begin{align}
		G_u(w_1^h,w_2^h) & = (a+c)g(w_1,w_2) +d \ g(w_1,u)g(w_2,u),\notag\\
		G_u(w_1^h,{w'_2}^v) & = bg(w_1,w'_2),\label{eq:g-natural-unit-sphere}\\
		G_u({w'_1}^v,{w'_2}^v) & = ag(w'_1,w'_2),\notag
	\end{align}
	where $ w_1,w_2 \in T_pM,  w'_1,w'_2 \in \{u\}^\perp , $ and $a,b,c,d \in \bR. $
	Therefore, $ G $ is 
	\begin{enumerate}
		\item non-degenerate if and only if $ a(a+c) - b^2 \neq 0 $ and $ a+c+d\neq 0, $
		\item Riemannian if and only if $ a(a+c) - b^2 > 0, a>0 $ and $ a+c+d > 0. $
	\end{enumerate}
	We call $ G $ a Kaluza-Klein type metric on a unit sphere bundle if $ b=0, $ and $ a(a+c)(a+c+d) \neq 0. $ Moreover, Kaluza-Klein metrics on a unit sphere bundle form a two-parameter family of $ g $-natural metrics for $ b=d=0, $ and $ a(a+c) \neq 0. $  
	
	%%%--------New subsection	
	\subsection{Metrics of Kaluza-Klein type on $ \AdS_3 $}\label{sc:Metric of Kaluza-Klein type}
	Let $ \bR^{2,2} $ denote the four-dimensional pseudo-Euclidean space equipped with the standard metric $ g_0 = -(dt^1)^2 -(dt^2)^2 +(dx^1)^2 + (dx^2)^2. $  
	A hypersurface $ \iota \colon \bH^3_1(\kappa) \hookrightarrow \bR^{2,2} $ endowed with the induced metric $ g=\iota^*g_0 $ is a Lorentzian manifold of negative constant curvature, known as the three-dimensional anti-de Sitter spacetime $ (\AdS_3). $ In other words,
	\begin{equation}
		\bH^3_1(\kappa) =\lc (t^1,t^2,x^1,x^2) \in \bR^{2,2}\mid -|t|^2+|x|^2 = -\frac{1}{\kappa}(\kappa>0)\rc.  
	\end{equation}
	
	The three-demensional anti-de Sitter spacetime is diffeomorphic to $ \bS^1\times \bR^2. $
	We describe a covering map from $ \bH^3_1(\kappa/4) $ to $ T^1\bH^2(\kappa) \simeq \bS^1\times \bR^2 $ using the following complex coordinates:
	\begin{equation}
		\bH^3_1 \coloneqq \bH^3_1(1) = \{ (z^1,z^2) \in \bC^2 \mid |z^1|^2 - |z^2|^2 =1 \}.
	\end{equation}
	We identify $ \bH^3_1 $ with the Lie group $ \SU(1,1) $ via the map:
	\begin{equation}\label{eq:isom-H^3_1-SU(1,1)}
		\bH^3_1 \ni (z^1,z^2) \mapsto \begin{pmatrix} z^1 & z^2 \\ \overline{z}^2 & \overline{z}^1 \end{pmatrix}\in \SU(1,1).
	\end{equation}
	The Lie group $ \SU(1,1), $ which is identified with $ \SL(2,\bR), $ is diffeomorphic to the non-Riemannian three-dimensional spin group $ \Spin(1,2). $
	The double cover $ \Spin(1,2)\simeq \SU(1,1) \ni a \mapsto \Ad_{a} \in \SO_0(1,2) $ is given by
	\begin{equation}
		a = \begin{pmatrix} z^1 & z^2 \\ \overline{z}^2 & \overline{z}^1 \end{pmatrix} \mapsto \begin{pmatrix}|z^1|^2 + |z^2|^2 & 2\Re( -\sqrt{-1}\overline{z}^1z^2)  & 2 \Im(-\sqrt{-1}\overline{z}^1z^2) \\ -2\Re( \sqrt{-1}z^1z^2)& \Re((z^1)^2-(z^2)^2) & \Im((\zbar^1)^2+(\zbar^2)^2)\\ -2\Im(\sqrt{-1}z^1z^2) & \Im((z^1)^2-(z^2)^2) & \Re((\zbar^1)^2+(\zbar^2)^2) \end{pmatrix} \eqqcolon A_{(z^1,z^2)}.
	\end{equation}	
	The matrix representation of the double cover is given with respect to the following orthonormal basis of Pauli-like matrices:
	\begin{equation} \label{eq:su(1,1)-basis}
		X_0 = \begin{pmatrix} \sqrt{-1} & 0 \\ 0 & -\sqrt{-1} \end{pmatrix}, X_1 = \begin{pmatrix} 0 & 1 \\ 1 & 0 \end{pmatrix}, X_2 = \begin{pmatrix} 0 & \sqrt{-1} \\ -\sqrt{-1} & 0 \end{pmatrix} \subset \mathfrak{su}(1,1),
	\end{equation}
	which is identified with the left-invariant vector fields on $ \SU(1,1) \simeq \bH^3_1. $
	Then the orthonormal frame on $ (\bH^3_1(\kappa/4), \iota^*g_0) $ is $ \lc \frac{\sqrt{\kappa}}{2}X_\alpha \rc_{\alpha=0}^2. $
	The left-invariant vector fields of $\{X_\alpha\}_{\alpha=0}^2 $ are given by 
	\begin{align}\label{eq:infinitesimal-action}
		X_0|_{(z^1,z^2)}  & = \sqrt{-1}z^1\frac{\partial}{\partial z^1} - \sqrt{-1}\overline{z}^1\frac{\partial}{\partial \overline{z}^1} - \sqrt{-1}z^2\frac{\partial}{\partial z^2} + \sqrt{-1}\overline{z}^2\frac{\partial}{\partial \overline{z}^2}\notag\\
		& = -t^2\frac{\partial}{\partial t^1} + t^1\frac{\partial}{\partial t^2} + x^2\frac{\partial}{\partial x^1} - x^1\frac{\partial}{\partial x^2}, \notag\\
		X_1|_{(z^1,z^2)} & = z^2 \frac{\partial}{\partial z^1} + \zbar^2 \frac{\partial}{\partial \zbar^1} + z^1 \frac{\partial}{\partial z^2} + \zbar^1 \frac{\partial}{\partial\overline{z}^2}\\
		& = x^1\frac{\partial}{\partial t^1} + x^2\frac{\partial}{\partial t^2} + t^1\frac{\partial}{\partial x^1} + t^2\frac{\partial}{\partial x^2}, \notag\\
		X_2|_{(z^1,z^2)}  & = -\sqrt{-1}z^2\frac{\partial}{\partial z^1} + \sqrt{-1}\zbar^2\frac{\partial}{\partial \overline{z}^1} +\sqrt{-1}z^1\frac{\partial}{\partial z^2} - \sqrt{-1}\zbar^1\frac{\partial}{\partial \overline{z}^2}\notag\\
		& = x^2\frac{\partial}{\partial t^1} - x^1\frac{\partial}{\partial t^2} - t^2\frac{\partial}{\partial x^1} +t^1\frac{\partial}{\partial x^2}.\notag
	\end{align}
	
	On the other hand, the Riemannian hyperbolic two-space $ \bH^2(\kappa) $ is embedded in the pseudo-Euclidean space $ \bR^{1,2} $ as follows:
	\begin{equation}
		\bH^2(\kappa)=\lc (t,y^1,y^2) \in\bR^{1,2} \mid -t^2+(y^1)^2 + (y^2)^2 = -\frac{1}{\kappa}, t>0\rc.
	\end{equation}
	We define the map $ F \colon \bH^3_1(\kappa/4) \to T^1\bH^2(\kappa) $
	as the composition of the following maps:
	\begin{align}
		& \bH^3_1(\kappa/4) \ni \frac{2}{\sqrt{\kappa}}(z^1,z^2)  \mapsto (z^1,z^2) \in \bH^3_1, \\
		& \bH^3_1 \ni (z^1,z^2)  \mapsto \begin{pmatrix}z^1 & z^2 \\ \overline{z}^2 & \overline{z}^1	\end{pmatrix} \in \SU(1,1), \\
		& \SU(1,1) \ni \begin{pmatrix}z^1 & z^2 \\ \overline{z}^2 & \overline{z}^1	\end{pmatrix} \mapsto A_{(z^1,z^2)} = (c_0,c_1,c_2) \in \SO_0(1,2) , \\
		& \SO_0(1,2) \ni (c_0,c_1,c_2) \mapsto \lb \frac{1}{\sqrt{\kappa}}c_0, c_1 \rb \in T_1\bH^2(\kappa).
	\end{align}	
	Therefore, the following composition of the maps is an isomorphism.
	\begin{equation}\label{eq:Hopf-map}
		F \colon \bH^3_1(\kappa/4) \ni \frac{2}{\sqrt{\kappa}}(z^1,z^2)  \mapsto \lb\frac{1}{\sqrt{\kappa}}(|z^1|^2+|z^2|^2, -2\sqrt{-1}z^1z^2), (2\Im(\zbar^1z^2), (z^1)^2-(z^2)^2) \rb \in T^1 \bH^2(\kappa).
	\end{equation}
	Since $ \pi \circ F \colon \bH^3_1(\kappa/4) \ni \frac{2}{\sqrt{\kappa}}(z^1,z^2)  \mapsto \frac{1}{\sqrt{\kappa}}(|z^1|^2+|z^2|^2, -2\sqrt{-1}z^1z^2 ) \in \bH^2(\kappa) $ gives a pseudo-Riemannian submersion called a timelike Hopf fibration described in \pref{sc:Hopf-fibration}, we introduce the timelike Hopf coordinates:
	\begin{align}
		z^1 & = \sqrt{-1}\exp\lb \frac{\sqrt{-1}(\varphi+\tau)}{2}\rb\cosh \lb \frac{\chi}{2}\rb, \label{eq:timelike-Hopf coordinate1} \\
		z^2 & = \exp\lb \frac{\sqrt{-1}(\varphi-\tau)}{2}\rb\sinh \lb \frac{\chi}{2}\rb.\label{eq:timelike-Hopf coordinate2}
	\end{align}
	Let $ \{e_\alpha\}_{\alpha=0}^2 $ be an orthonormal frame of this coordinates on $ \bH^3_1(\kappa/4) $ that satisfies $ \iota_*e_\alpha = \frac{\sqrt{\kappa}}{2}X_\alpha (\alpha=0,1,2), $ i.e.,
	\begin{align}
		e_0 & = \sqrt{\kappa}\frac{\partial}{\partial \tau},\label{eq:e0}\\
		e_1 & = \sqrt{\kappa} \lc \frac{\cos{\tau}}{\sinh{\chi}}\lb \frac{\partial}{\partial \varphi} - \cosh\chi\frac{\partial}{\partial \tau} \rb - \sin{\tau}\frac{\partial}{\partial \chi}  \rc,\label{eq:e1}\\
		e_2 & = \sqrt{\kappa} \lc -\frac{\sin{\tau}}{\sinh{\chi}}\lb \frac{\partial}{\partial \varphi} - \cosh\chi\frac{\partial}{\partial \tau} \rb- \cos{\tau}\frac{\partial}{\partial \chi}  \rc.\label{eq:e2}
	\end{align}
	In this case,
	\begin{equation}\label{eq:standard-AdS3-timelike-Hopf-coordinates}
		g = \iota^*g_0 = \dfrac{1}{\kappa}\lc-\lb d\tau+\cosh{\chi}d\varphi \rb^2 +d\chi^2+\sinh^2{\chi}d\varphi^2 \rc
	\end{equation}
	is the Lorentzian metric on $ \bH^3_1(\kappa/4) $ expressed in the timelike Hopf coordinates.
	The timelike Hopf fibration in the timelike Hopf coordinates is given by the following:
	\begin{equation}
		\pi\circ F \colon (\tau, \chi, \varphi) \mapsto (t,y^1,y^2) =\frac{1}{\sqrt{\kappa}} (\cosh{\chi}, \sinh{\chi}\cos{\varphi}, \sinh{\chi}\sin{\varphi}).
	\end{equation}
	Then the induced metric $ h_0  $ on the base manifold $ \bH^2(\kappa) $ of this coordinates is given by
	\begin{equation}\label{eq:standard-H2-coordinates}
		h_0 = \frac{1}{\kappa}\lb d\chi^2+\sinh^2{\chi}d\varphi^2\rb.
	\end{equation}
	We obtain the following formulae \eqref{eq:lifting-formulae} in \pref{lm:lifting-formulae} for the pushforward of the left-invariant vector fields $ \lc \frac{\sqrt{\kappa}}{2}X_\alpha \rc_{\alpha=0}^2 $ through the following direct calculations:
	\begin{align}
		\frac{\sqrt{\kappa}}{2}F_* X_0|_{(z^1,z^2)} &= -\sqrt{\kappa}\Re(\zbar^1z^2) \frac{\partial}{ \partial \xibar^0} -\frac{\sqrt{\kappa}}{2}\Im((z^1)^2+(z^2)^2)\frac{\partial}{\partial \xibar^1} + \frac{\sqrt{\kappa}}{2}\Re((z^1)^2 + (z^2)^2)\frac{\partial}{\partial \xibar^2},\label{eq:X0-Killing-embedding}\\
		\frac{\sqrt{\kappa}}{2}F_* X_1|_{(z^1,z^2)} &= 2\Re(\zbar^1z^2) \frac{\partial}{\partial t} + \Im((z^1)^2+(z^2)^2)\frac{\partial}{\partial y^1} - \Re((z^1)^2+(z^2)^2)\frac{\partial}{\partial y^2}, \label{eq:X1-Killing-embedding}  \\
		\frac{\sqrt{\kappa}}{2}F_* X_2|_{(z^1,z^2)} &= 2\Im(\zbar^1z^2) \frac{\partial}{\partial t} + \Re((z^1)^2-(z^2)^2)\frac{\partial}{\partial y^1} + \Im((z^1)^2-(z^2)^2)\frac{\partial}{\partial y^2} \label{eq:X2-Killing-embedding}\\
		&+ \frac{\sqrt{\kappa}}{2}(|z^1|^2+|z^2|^2)\frac{\partial}{\partial \xibar^0} +\sqrt{\kappa}\Im(z^1z^2)\frac{\partial}{\partial \xibar^1} -\sqrt{\kappa}\Re(z^1z^2)\frac{\partial}{\partial \xibar^2},\notag
	\end{align}
	where $ (\xibar^0,\xibar^1,\xibar^2) $ are the coordinates of the fiber of the tangent bundle $ T\bR^{1,2}. $
	\begin{lemma}\label{lm:lifting-formulae}\cite{benyounes2011generalized}
		Let $ u = \frac{\sqrt{\kappa}}{2}\pi_*F_*X_2|_{(z^1,z^2)} $ and $ w = \frac{\sqrt{\kappa}}{2}\pi_*F_*X_1|_{(z^1,z^2)} \in T_{(z^1,z^2)}\bH^2(\kappa) $ be tangent vectors. 
		The following results about the liftings with respect to the tangent vector $ u \in T_{(z^1,z^2)}\bH^2(\kappa) $ describe the relations between the tangent vectors on $ \bH^2(\kappa) $ and $ T^1\bH^2(\kappa): $
		\begin{align}
			F_* X_0|_{(z^1,z^2)} & = - w^v, \notag\\
			\frac{\sqrt{\kappa}}{2}F_* X_1|_{(z^1,z^2)} & = w^h, \label{eq:lifting-formulae}\\
			\frac{\sqrt{\kappa}}{2}F_* X_2|_{(z^1,z^2)} & = u^h. \notag
		\end{align}
	\end{lemma}
	Let $ \{\theta^\alpha\}_{\alpha=0}^2 $ be the left-invariant dual frame of $ \lc \frac{\sqrt{\kappa}}{2}X_\alpha \rc_{\alpha=0}^2 $ on $ \bH^3_1(\kappa/4). $
	Now, we obtain the following result in \cite{calvaruso2014metrics}.
	We identify a metric of Kaluza-Klein type on $ T^1\bH^2(\kappa) $ with the induced metric on $ \bH^3_1(\kappa/4). $ 
	\begin{theorem}\cite{calvaruso2014metrics}
		The covering map $ F\colon \bH^3_1(\kappa/4) \to T^1\bH^2(\kappa) $ establishes a one-to-one correspondence between a pseudo-Riemannian  metric $ g_{\lambda\mu\nu} $ on $ \bH^3_1(\kappa/4) $ of the form 
		\begin{equation}\label{eq:Kaluza-Klein-type}
			g_{\lambda\mu\nu} = \lambda\theta^0\otimes\theta^0+\mu\theta^1\otimes\theta^1+\nu\theta^2\otimes\theta^2,
		\end{equation}
		and a metric $ G $ of Kaluza-Klein type on $ T^1\bH^2(\kappa) $ as descrived in \eqref{eq:g-natural-unit-sphere}, defined by the following conditions:
		\begin{equation}
			a = \frac{4\lambda}{\kappa}, \quad b=0, \quad c=\mu-\frac{4\lambda}{\kappa}, \quad d=\nu-\mu.
		\end{equation}
	\end{theorem}
	\begin{definition}\label{df:Kaluza-Klein-type}
		A pseudo-Riemannian metric $ g $ on $ \bH^3_1 = \bH^3_1(1) $ is said to be Kaluza-Klein type if there exist three real constants $ \lambda,\mu,\nu \neq 0, $ such that $ g = g_{\lambda\mu\nu} $ is described in \eqref{eq:Kaluza-Klein-type}.
	\end{definition}
	 Hereafter, we assume that $\lambda \neq 0$ and $\mu, \nu > 0$. For simplicity, only the sign of $\lambda$ determines whether the metric $ g_{\lambda\mu\nu} $ is Riemannian or Lorentzian, since $ \bS^1 $ direction is the standard timelike coordinate of $ \AdS_3. $
	Next, we discuss the homogeneous structures and the group of isometries with the reductive decompositions for the each Kaluza-Klein type metric $ g_{\lambda\mu\nu} $ on $ \bH^3_1 $.

	%%%--------New section	
	\section{Homogeneous pseudo-Riemannian structures on $ (\bH^3_1, g_{\lambda\mu\nu}) $ }
	\label{sc:HS-classification}
	If $ -\lambda=\mu=\nu=1,  (\bH^3_1, g_{\lambda\mu\nu}) $ represents the hypersurface in $ \bR^{2,2} $ with the standard Lorentzian metric $ \iota^*g_0 $ known as $ \AdS_3. $
	In this case, $ (\bH^3_1, g_{\lambda\mu\nu}), $ which is symmetric, has a homogeneous Lorentzian structure $ S=0 $.
	The group of isometries of $ \AdS_3 $ is $ \SO_0(2,2)\simeq (\SU(1,1)\times \SU(1,1))/\bZ_2. $
	In this section, we study all homogeneous pseudo-Riemannian structures (homogeneous structures) $ S $ on $ (\bH^3_1,g_{\lambda\mu\nu}). $  
	Then every homogeneous pseudo-Riemannian  structure on $(\bH^3_1, g_{\lambda\mu\nu}) $ has an isometry group that acts transitively and almost effectively of the form $ \SU(1,1)\times H $ as
	\begin{equation}\label{eq:isometric action-subgroup}
		\SU(1,1)\times H \times\bH^3_1 \ni \lb (g_L, g_R), Z \rb \mapsto g_L Zg_R^\dagger \in \bH^3_1, 
	\end{equation}
	where $ H =U(1), \bR, \Aff(\bR), \SU(1,1) $ or $ \{e\} $ is a subgroup of $ \SU(1,1). $
	
	We shall calculate a $ (0,3) $-type  tensor field $ S $ called a homogeneous pseudo-Riemannian structure that is given by the following expression for each $ g_{\lambda\mu\nu} $ from the condition $\widetilde\nabla g_{\lambda\mu\nu} = 0$:
	\begin{equation}\label{eq:homogeneous-str}
		S = \rho\otimes\theta^0\wedge\theta^1 + \sigma\otimes\theta^1\wedge\theta^2+\tau\otimes\theta^2\wedge\theta^0.
	\end{equation} 
	The corresponding $(1,2)$-tensor field with the metric isomorphism \eqref{eq:homogeneous structure tensor} defines the canonical connection $ \widetilde\nabla = \nabla-S, $ which is a metric connection. 
	We denote $ \rho_\alpha, \sigma_\alpha, $ and $ \tau_\alpha $ as the functions on $ \bH^3_1 $ for $ \alpha=0,1,2 $ and we express $ \rho, \sigma $ and $\tau $ in $ \mathfrak{su}(1,1)^* $ as follows: \begin{equation}
		\rho=\rho_\alpha \theta^\alpha, \quad \sigma=\sigma_\alpha\theta^\alpha,\quad \tau=\tau_\alpha\theta^\alpha. 
	\end{equation}
	The Lie algebra structure of the left-invariant vector fields $ \{X_\alpha\}_{\alpha=0}^2 $ on $ \bH^3_1 $ defined in \eqref{eq:su(1,1)-basis} is the following:
	\begin{equation}\label{eq:Lie-alg-su(1,1)}
		[X_0,X_1]=2X_2, \quad [X_1,X_2]=-2X_0, \quad [X_2,X_0]=2X_1.
	\end{equation}
	The Levi-Civita connection $ \nabla $ of $ g_{\lambda\mu\nu} $ is given by
	\begin{equation}
		2g_{\lambda\mu\nu}(\nabla_XY,Z) = g_{\lambda\mu\nu}([X,Y],Z) - g_{\lambda\mu\nu}([Y,Z],X) + g_{\lambda\mu\nu}([Z,X],Y),
	\end{equation}
	for all $ X,Y,Z\in \mathfrak{su}(1,1). $
	Therefore, the connection forms $ \{\omega^i_j\} $ with resspect to the left-invariant frame of $ \nabla $ are given as follows:
	\begin{align}
		\omega^0_1 & = \frac{\lambda-\mu+\nu}{\lambda}\theta^2 ,\ \omega^0_2 = -\frac{\lambda+\mu-\nu}{\lambda}\theta^1, \notag\\
		\omega^1_0 & = -\frac{\lambda-\mu+\nu}{\mu}\theta^2, \omega^1_2 = -\frac{\lambda+\mu+\nu}{\mu}\theta^0, \label{eq:connection-form-general}\\
		\omega^2_0 & = \frac{\lambda+\mu-\nu}{\nu}\theta^1, \ \omega^2_1 = \frac{\lambda+\mu+\nu}{\nu}\theta^0.\notag
	\end{align}	 
	The non-trivial components of the curvature tensor field are as follows:
	\begin{align}
		R(X_0,X_1) & = \lb \frac{\lambda}{\nu} + \frac{2 \mu}{\nu} - 2 + \frac{\mu^{2}}{\lambda \nu} + \frac{2 \mu}{\lambda} - \frac{3 \nu}{\lambda}\rb \theta^1\otimes X_0 \notag\\
		& + \lb - \frac{\lambda^{2}}{\mu \nu} - \frac{2 \lambda}{\nu} + \frac{2 \lambda}{\mu} - \frac{\mu}{\nu} - 2 + \frac{3 \nu}{\mu}\rb \theta^0\otimes X_1,\notag\\
		R(X_1,X_2) & = \lb - \frac{3 \lambda}{\mu} - 2 - \frac{2 \nu}{\mu} + \frac{\mu}{\lambda} - \frac{2 \nu}{\lambda} + \frac{\nu^{2}}{\lambda \mu} \rb \theta^2\otimes X_1 \label{eq:curvature-form-general} \\
		& + \lb \frac{3 \lambda}{\nu} + \frac{2 \mu}{\nu} + 2 - \frac{\mu^{2}}{\lambda \nu} + \frac{2 \mu}{\lambda} - \frac{\nu}{\lambda} \rb \theta^1\otimes X_2,\notag\\
		R(X_2,X_0) & = \lb - \frac{\lambda}{\mu} + 2 - \frac{2 \nu}{\mu} + \frac{3 \mu}{\lambda} - \frac{2 \nu}{\lambda} - \frac{\nu^{2}}{\lambda \mu} \rb \theta^2 \otimes X_0 \notag \\
		& +\lb\displaystyle \frac{\lambda^{2}}{\mu \nu} - \frac{2 \lambda}{\nu} + \frac{2 \lambda}{\mu} - \frac{3 \mu}{\nu} + 2 + \frac{\nu}{\mu} \rb \theta^0\otimes X_2.\notag
	\end{align}
	Then under the conditions $ \widetilde\nabla S = 0 $ and $ \widetilde\nabla R = 0, $ we determine the homogeneous structure $ S. $
	
	Conversely, we define an appropriate transitive action, which is given by \eqref{eq:isometric action-subgroup}, on $ \bH^3_1 $ whose Lie algebra structure corresponds to the $ (\nabla-S)_{X^*}Y^* = -[X,Y]^* (X,Y \in \frakm = T_oM) $ for each homogeneous structure $ S. $
	First, we have to confirm whether the metric $ g_{\lambda\mu\nu} $ on $ \bH^3_1 $ is invariant under the isometric action or not. The following \pref{pr:invariant-metric} gives a description of invariant metrics on a reductive homogeneous space and its $ \Ad|_H $-invariant subspace.
	
	\begin{proposition}\label{pr:invariant-metric}\cite{arvanitogeorgos2003introduction}
		Let $ M = G/H $ be a reductive homogeneous space with a reductive decomposition $ \frakg=\frakh\oplus\frakm, $ that is, $ \frakg $ and $\frakh $ are Lie algebra of $ G $ and $ H, $ and $ \frakm $ is an $ \Ad|_H $-invariant subspace. 
		Then, there is a one-to-one correspondence between:
		\begin{enumerate}
			\item A $ G $-invariant pseudo-Riemannian metric $ g $ on $ M, $
			\item An $ \Ad|_H $-invariant linear pseudo-Riemannian metric $ \la \bullet , \bullet \ra $ on $ \mathfrak{m}. $
		\end{enumerate}
	\end{proposition}  
	
	This \pref{pr:invariant-metric} represents there are no deformed metric $ g_{\lambda\mu\nu} $ on $ \bH^3_1 $ from the standard Lorentzian metric of $ \AdS_3 $ along any lightlike Killing.
	
	%%%-----------New subsection different	
	\subsection{The case of \( -\lambda\neq\mu,\mu\neq\nu, \nu \neq -\lambda \)}\label{sc:trivial}
	From \pref{pr:invariant-metric}, the all corresponding isotropy groups are trivial.
	The connenction form $ \omega $ and curvature tensor $ R $ of the Levi-Civita connection $ \nabla $ were given in \eqref{eq:connection-form-general}, \eqref{eq:curvature-form-general}.
	We have the following equations from $ (\nabla-S)R=0: $
	\begin{align}
		\rho_0(\lambda-\mu+\nu)& =0,\\
		\rho_1(\lambda-\mu+\nu) & = 0,\\
		\rho_2(\lambda-\mu+\nu) & = - (\lambda-\mu+\nu)^2,\\
		\sigma_0(\lambda+\mu+\nu) & = (\lambda+\mu+\nu)^2,\\
		\sigma_1(\lambda+\mu+\nu) & = 0,\\
		\sigma_2(\lambda+\mu+\nu) & = 0,\\
		\tau_0(\lambda+\mu-\nu) & = 0,\\
		\tau_1(\lambda+\mu-\nu) & = -(\lambda+\mu-\nu)^2,\\
		\tau_2(\lambda+\mu-\nu) & = 0.
	\end{align}
	As a result of the above equations and $ (\nabla-S)S = 0, $ the homogeneous structures $ S=S_0 $ is given by the following: 
	\begin{equation}\label{eq:homogeneous-str-different}
		S_0 = -(\lambda-\mu +\nu)\theta^2\otimes\theta^0\wedge\theta^1 + (\lambda + \mu + \nu)\theta^0\otimes\theta^1\wedge\theta^2 -(\lambda + \mu -\nu) \theta^1\otimes\theta^2\wedge\theta^0.
	\end{equation}
	When $S=S_0$, the canonical connection $ \widetilde\nabla= \nabla-S $ has a trivial connection form with respect to $\{X_\alpha\}_{\alpha=0}^2 $; hence, the curvature tensor vanishes.
	Therefore, we obtain the following result.
	
	\begin{theorem}\label{th:homogeneous-structure-different}
		Assume that $ - \lambda \neq \mu, \mu \neq \nu, \nu \neq -\lambda.$ 
		There exists a transitive and effective action of the isometry group $ \SU(1,1) $ on pseudo-Riemannian manifolds $ (\bH^3_1, g_{\lambda\mu\nu}) $ from the left induced by the homogeneous structure $S=S_0$ in \eqref{eq:homogeneous-str-different} with the trivial reductive decomposition.
	\end{theorem} 
%	\begin{remark}\label{rm:HS-different-coincidence}
%		The homogeneous structure $S=S_0 $ in \eqref{eq:homogeneous-str-different} is coincide with $ S_\lambda(t) $ and $ S_\mu(t) $ later	 if and only if $t=\lambda+2\mu $ for $S_0 = S_\lambda(t)(\mu=\nu) $ and $t = 2\nu-\mu $ for $ S_0 = S_\mu(t)(-\lambda=\nu).$ In any case of $ S=S_0, $ the isotropic subgroup of the isometries is $ \{e\}. $
%	\end{remark}
	
	%%%--------- New subsection
	
	\subsection{The case of \( - \lambda \neq \mu = \nu \)}\label{sc:time}
	In this case, the metric $ g_{\lambda\mu\nu} $ is deformed along the timelike Hopf fiber $ \bS^1 \hookrightarrow \bH^3_1 \to \bH^2(4), $ where we choose the direction of Hopf fiber along to $ X_0. $
	We obtain the following result regarding the existence and classification of homogeneous structures on $ (\bH^3_1, g_{\lambda\mu\nu}). $
	\begin{theorem}\label{th:homogeneous-str-timelike}
		If $ -\lambda \neq \mu=\nu, $ any homogeneous structure $ S $ on the pseudo-Riemannian manifold $ (\bH^3_1, g_{\lambda\mu\nu}) $ is given by the following $ S = S_\lambda(t) $ or $ S_0 : $ 
		\begin{equation}
			S_\lambda(t) = -\lambda\theta^2\otimes\theta^0\wedge\theta^1 + t\theta^0\otimes\theta^1\wedge\theta^2-\lambda\theta^1\otimes\theta^2\wedge\theta^0, \label{eq:s1}
		\end{equation}
		where $ \lambda + 2\mu \neq t \in \bR$ is a constant. Moreover, $ S_\lambda(t) $ coincides with $ S_0, $ by substituting $ t = \lambda+2\mu.$
	\end{theorem}
	\begin{proof}
		We have the following conditions by $ \widetilde\nabla R = 0 $ that the homogeneous structures $ S $ should be satisfied:
		\begin{equation}
			\tau_0 = \rho_0 = \rho_1 = \tau_2 = 0,\quad \tau_1= \rho_2 = -\lambda.
		\end{equation}
		Then we have
		$ S = -\lambda\theta^2\otimes\theta^0\wedge\theta^1 + \sigma\otimes\theta^1\wedge\theta^2-\lambda\theta^1\otimes\theta^2\wedge\theta^0, $
		and the following equations
		\begin{align}
			X_\alpha(\sigma_0) & = 0,(\alpha=0,1,2)\\
			\mu X_0(\sigma_1) & = -\sigma_2(\sigma_0-\lambda-2\mu),\\
			\mu X_0(\sigma_2) & = \sigma_1(\sigma_0-\lambda-2\mu),\\
			\mu X_1(\sigma_1) & = -\sigma_1\sigma_2,\\
			\mu X_1(\sigma_2) & = \sigma_1^2,\\
			\mu X_2(\sigma_1) & = -\sigma_2^2,\\
			\mu X_2(\sigma_2) & = \sigma_1\sigma_2,
		\end{align}
		are derived from $ \widetilde\nabla S = 0. $
		Then the coefficients of homogeneous structure expanded by left-invariant forms are constants.
		Therefore, we obtain $ S = S_\lambda(t) $ by setting $ \sigma_0 = t $ as a constant.
	\end{proof}
	Then we determine the action of the isometry group of the form $\SU(1,1) \times H $ and the reductive decomposition induced by the homogeneous structure $S_\lambda(t)$.
	\begin{theorem}\label{th:isometry-timelike-1}
		Assume that $ -\lambda \neq \mu=\nu. $  
		If $ t \neq \lambda+2\mu, $
		the isometry group $ \SU(1,1)\times U(1) $ of the pseudo-Riemannian manifold $(\bH^3_1,g_{\lambda\mu\nu})$ acts transitively and almost effectively on it via the map:
		\begin{equation}\label{eq:U(1,1)-action}
			\phi\colon \SU(1,1) \times U(1) \times \bH^3_1 \ni \lb (A_0, e^{\sqrt{-1}s}), Z \rb \mapsto A_0 Z 
			\begin{pmatrix}
				e^{-\sqrt{-1}s} & 0 \\
				0 & e^{\sqrt{-1}s}
			\end{pmatrix} 
			\in \bH^3_1.
		\end{equation}
 This leads to a transitive and effective isometric action of $ U(1,1), $ with the following reductive decomposition $ \frakg = \frakm \oplus \frakh $ induced by the homogeneous structure $ S = S_\lambda(t): $
		\begin{align}
			\frakh & = \vspan_\bR\lc 
			\begin{pmatrix}
				0 & 0 \\
				0 & \sqrt{-1}
			\end{pmatrix} \rc,\\ 
			\frakm & = \vspan_\bR
			\lc 
			\begin{pmatrix}
				\sqrt{-1} & 0 \\
				0 & \lb1+\frac{\lambda - t}{\mu}\rb \sqrt{-1}
			\end{pmatrix},
			\begin{pmatrix}
				0 & 1 \\
				1 & 0
			\end{pmatrix},
			\begin{pmatrix}
				0 & \sqrt{-1} \\
				-\sqrt{-1} & 0
			\end{pmatrix}
			\rc.
		\end{align}
		
		If $ t = \lambda+2\mu, \ \SU(1,1) $ trasitively and effectively acts on $ \bH^3_1 $ as isometry group from the left with the trivial reductive decomposition. 
	\end{theorem}
	\begin{remark}
		We identify $ \bH^3_1 $ with $ \SU(1,1) = \lc
		\begin{pmatrix}
			z^1 & z^2 \\ \zbar^2 & \zbar^1	
		\end{pmatrix} \mid (z^1,z^2) \in \bC^2, |z^1|^2 - |z^2|^2 = 1 \rc $ as in \eqref{eq:isom-H^3_1-SU(1,1)}.
		The isotropy group of $ \phi $ at $ o = (1,0)\in \bH^3_1 $ is the following:
		\begin{equation}
			H = \Stab(o)  = \lc \lb
			\begin{pmatrix}
				e^{\sqrt{-1}s} & 0 \\ 0  & e^{-\sqrt{-1}s}
			\end{pmatrix}, e^{\sqrt{-1}s} \rb \mid s \in \bR \rc. 
		\end{equation}
		Since an element $ \lb
		\begin{pmatrix}
			-1 & 0 \\ 0  & -1
		\end{pmatrix}, -1 \rb \in \SU(1,1)\times U(1) $ acts on $ \bH^3_1 $ by $ \phi $ as an identity, this action is not effective.
	\end{remark}
	\begin{proof}
		From \pref{pr:invariant-metric}, in this case, the $ g_{\lambda\mu\nu} $ is $ G $-invariant if and only if $ G =  (\SU(1,1) \times U(1))/\Gamma \simeq U(1,1) $ under the transitive and effective isometric action induced by $ \phi $. 
		
		The non-trivial components of the connection form and the curvature tensor of the canonical connection $ \widetilde\nabla = \nabla - S_\lambda(t) $ are the following:
		\begin{align}
			\omegatilde^1_2 & = -\frac{\lambda +2\mu -t}{\mu}\theta^0 = -\omegatilde^2_1, \label{eq:nablatilde-timelike-1}\\
			\Rtilde(X_1,X_2) & = \frac{2(\lambda+2\mu-t)}{\mu}\lb \theta^1\otimes X_2 -\theta^2\otimes X_1 \rb.\label{eq:Rtilde-timelike}
		\end{align}
		The holonomy Lie algebra of this connection is 
		$ \widetilde\frakh = \vspan_\bR\lc \Rtilde(X_1,X_2) \rc. $
		Hence it vanishes when $ S=S_0 $ (i.e., when $ t=\lambda+2\mu). $ It implies that the isotropy group of the isometries is trivial. 
		
		Hereafter, we assume that  $ t \neq \lambda+2\mu $, then it does not vanish. If we set $ U = \frac{\mu}{\lambda+2\mu-t}\Rtilde(X_1,X_2) = 2(\theta^1\otimes X_2-\theta^2\otimes X_1) $ and $ S = S_\lambda(t), $
		the Lie algebra satisfies the following conditions:
		\begin{align}
			[X_0,X_1] & = S_{X_0}X_1 - S_{X_1}X_0 = \frac{t-\lambda}{\mu}X_2, \\
			[X_1,X_2] & =  S_{X_1}X_2 - S_{X_2}X_1 - \Rtilde(X_1,X_2) = -2X_0 - \frac{\lambda+2\mu-t}{\mu}U, \\
			[X_2,X_0] & =  S_{X_2}X_0 - S_{X_0}X_2 = \frac{t-\lambda}{\mu}X_1,
		\end{align}
		\begin{equation}
			[U,X_0] = 0,\quad [U,X_1] = 2X_2,\quad [U,X_2] = -2X_1.
		\end{equation}
		Then the Lie algebra relations of $ \widetilde\frakg = \vspan_\bR\{\Uhat\}\oplus\vspan_\bR\{\Xhat_\alpha\}_{\alpha=0}^2 $ are given by
		\begin{equation}
			[\Xhat_0, \Xhat_1]= 2\Xhat_2,\ [\Xhat_1, \Xhat_2]= -2\Xhat_0, \ [\Xhat_2, \Xhat_0] = 2\Xhat_1,\ [\Uhat, \Xhat_\alpha]=0,(\alpha=0,1,2)
		\end{equation}
		for
		\begin{align}
			\Uhat & = -\frac{\lambda-t}{2\mu}U -X_0,\label{eq:hatU}\\
			\Xhat_0 & = X_0 + \frac{\lambda-t+2\mu}{2\mu}U,\label{eq:hatX0}\\
			\Xhat_1 & = X_1,\label{eq:hatX1}\\
			\Xhat_2 & = X_2.\label{eq:hatX2}
		\end{align}
		
		On the other hand, we show that the connection $ \nabla - S_\lambda(t) $ coincides with the one described $ \widetilde\nabla_{X^*}Y^*|_{o}=-[X,Y]^*|_{o} $ for $ X,Y \in \frakm $ at $ o = (1,0) \simeq I_2 \in \SU(1,1), $ which is defined below, by the action $ \phi\colon \SU(1,1) \times U(1) \times \bH^3_1 \to \bH^3_1 $.  
		We also see the reductive decomposition $ \widetilde\frakg = \frakh\oplus\frakm. $
		The Lie algebra homomorphism $ \psi_* \colon \mathfrak{su}(1,1)\times \bR \to \fraku(1,1) $ induced by the Lie group homomorphism $ \psi \colon \SU(1,1)\times U(1) \ni (A, e^{\sqrt{-1}s}) \mapsto e^{-\sqrt{-1}s}A \in U(1,1)  $ is the following:
		\begin{align}
			\psi_*(\Xhat_0, 0) & = 
			\begin{pmatrix}
				\sqrt{-1} & 0 \\ 0 & -\sqrt{-1}
			\end{pmatrix}, \\
			\psi_*(\Xhat_1, 0) & = 
			\begin{pmatrix}
				0 & 1 \\ 1 & 0
			\end{pmatrix}, \\
			\psi_*(\Xhat_2, 0) & = 
			\begin{pmatrix}
				0 & \sqrt{-1} \\ -\sqrt{-1} & 0
			\end{pmatrix}, \\
			\psi_*(0, \Uhat) & = 
			\begin{pmatrix}
				-\sqrt{-1} & 0 \\ 0 & -\sqrt{-1}
			\end{pmatrix}.
		\end{align}
		Therefore, vectors
		\begin{align}
			U_0 &\coloneqq \psi_*(\Xhat_0,\Uhat)  = 
			\begin{pmatrix}
				0 & 0 \\ 0 & -2\sqrt{-1}
			\end{pmatrix},\\
			B_\alpha & \coloneqq \psi_*(\Xhat_\alpha,0),(\alpha =0,1,2)\\
			B_3 & \coloneqq \psi_*(0,\Uhat),
		\end{align}
		span $ \fraku(1,1) $.
		The induced Killing vector fields by $ B_0,B_1,B_2 \in \fraku(1,1)$ are generated by the isometric action $ \phi $ from the left.
		Whereas, the induced Killing vector field by $ B_3 \in \fraku(1,1) $ is generated by the right action.
		The direct sum decomposition of the Lie algebra
		$ \fraku(1,1) = \frakm \oplus \frakh, $ consisting of $ \frakh = \vspan_\bR\{U_0\}, $ and $ \frakm = \vspan_\bR\{E_t, B_1,B_2\} $ turns out to be the reductive decomposition via the isomorphism:
		\begin{equation}
			\tau \colon \frakm \ni B \mapsto B^*|_o \in T_{o}\bH^3_1,
		\end{equation}
		where, $ \tau(B_3) = -X_0^*|_{o}, \  \tau(B_\alpha) = X^*_\alpha|_{o},(\alpha=0,1,2) $
		and $ E_t = -\frac{2\mu+\lambda-t}{2\mu}B_3 -\frac{\lambda-t}{2\mu}B_0. $ i.e,
		\begin{equation}
			[U_0,E_t] = 0, \quad [U_0, B_1] = 2B_2, \quad [U_0, B_2] = -2B_1,
		\end{equation}
		\begin{equation}
			[E_t, B_1] = -\frac{\lambda-t}{\mu}B_2, \quad [B_1,B_2] = -2B_0, \quad [B_2,E_t] = -\frac{\lambda-t}{\mu}B_1.
		\end{equation}
		We have a connection $ \widetilde\nabla $ determined by the Killing vector fields generated by the elements in $ \frakm $ that is isomorphic to $ T_o\bH^3_1 $ by $ \tau $ as follows:
		\begin{align}
			\widetilde\nabla_{E_t^*}B_1^*|_{o} & = - [E_t,B_1]^*|_o = \frac{\lambda-t}{\mu}B_2^*|_o = \frac{\lambda-t}{\mu}X_2^*|_o ,\\
			\widetilde\nabla_{E_t^*}B_2^*|_{o} & = - [E_t,B_2]^*|_o = -\frac{\lambda-t}{\mu}B_1^*|_o = -\frac{\lambda-t}{\mu}X_1^*|_o ,\\
			\widetilde\nabla_{B_1^*}B_2^*|_{o} & = - [B_1,B_2]^*|_o = 2B_0^*|_o = 2X_0^*|_o.
		\end{align}
		Then we expand the left-invariant vector fields $ \{X_\alpha\}_{\alpha=0}^2 $ using the Killing vector fields generated by the elements in $ \frakm $ under the action $ \phi. $
		Let $ \{f_\alpha, g_\beta,h_\gamma\}_{\alpha,\beta,\gamma=t,1,2} \subset C^\infty(\bH^3_1)$ be coeficients as follows:
		\begin{align}
			X_0|_{(z^1,z^2)} & = f_tE^*_t|_{(z^1,z^2)}+f_1B^*_1|_{(z^1,z^2)}+f_2B^*_2|_{(z^1,z^2)},\\
			X_1|_{(z^1,z^2)} & = g_tE^*_t|_{(z^1,z^2)}+g_1B^*_1|_{(z^1,z^2)}+g_2B^*_2|_{(z^1,z^2)},\\
			X_2|_{(z^1,z^2)} & = h_tE^*_t|_{(z^1,z^2)}+h_1B^*_1|_{(z^1,z^2)}+h_2B^*_2|_{(z^1,z^2)}.
		\end{align}
		\begin{align}
			f_t(z^1,z^2) & = \frac{(|z^2|^2+|z^1|^2)\mu}{(\lambda-t+\mu)|z^2|^2+\mu|z^1|^2},  \notag\\
			f_1(z^1,z^2) & = \frac{\sqrt{-1}(\lambda-t)(z^1z^2-\zbar^1\zbar^2)}{2\{(\lambda-t+\mu)|z^2|^2+\mu|z^1|^2\}},\label{eq:f-time}\\
			f_2(z^1,z^2) & = \frac{(\lambda-t)(z^1z^2 + \zbar^1\zbar^2)}{2\{(\lambda-t+\mu)|z^2|^2+\mu|z^1|^2\}}.\notag
		\end{align}
		\begin{align}
			g_t(z^1,z^2) & = \frac{\sqrt{-1}(z^1\zbar^2-\zbar^1z^2)\mu}{(\lambda-t+\mu)|z^2|^2-\mu|z^1|^2},\notag\\
			g_1(z^1,z^2) & = \frac{(\lambda-t+\mu)\{(z^2)^2+(\zbar^2)^2\} + \mu\{(z^1)^2+(\zbar^1)^2\}}{2\{(\lambda-t+\mu)|z^2|^2+\mu|z^1|^2\}},\label{eq:g-time}\\
			g_2(z^1,z^2) & =-\frac{\sqrt{-1}\{(\lambda-t+\mu)\{(z^2)^2-(\zbar^2)^2\} + \mu\{(z^1)^2-(\zbar^1)^2\}\}}{2\{(\lambda-t+\mu)|z^2|^2+\mu|z^1|^2\}}.\notag
		\end{align}
		\begin{align}
			h_t(z^1,z^2) & = -\frac{(z^2\overline{z}^1+z^1\overline{z}^2)\mu}{(\lambda-t+\mu)|z^2|^2+\mu|z^1|^2},\notag\\
			h_1(z^1,z^2) & = -\frac{\sqrt{-1}\{(\lambda-t+\mu)\{(z^2)^2-(\zbar^2)^2\} - \mu\{(z^1)^2-(\zbar^1)^2\}\}}{2\{(\lambda-t+\mu)|z^2|^2+\mu|z^1|^2\}},\label{eq:h-time}\\
			h_2(z^1,z^2) & = -\frac{\mu\{(z^1)^2 + (\zbar^1)^2\}-(\lambda-t+\mu)\{(z^2)^2+(\zbar^2)^2\}}{2\{(\lambda-t+\mu)|z^2|^2+\mu|z^1|^2\}}.\notag
		\end{align}
		Then we calculate the connection $ \widetilde\nabla $ for the left-invariant vector fields $ \{X_\alpha\}_{\alpha=0}^2 $ at $ o \in \bH^3_1. $
		\begin{align}
			\widetilde\nabla_{X_0}X_1|_o & = E^*_t(g_t)B_0^*|_o + E^*_t(g_1)B_1^*|_o + E^*_t(g_2)B_2^*|_o + \widetilde\nabla_{E_t^*}B_1^*|_o \\
			& = 2B^*_2|_o + \frac{\lambda-t}{\mu}B^*_2|_o\notag \\
			& = \frac{\lambda-t+2\mu}{\mu}X_2|_o,\notag\\
			\widetilde\nabla_{X_0}X_2|_o & = E^*_t(h_t)E_t^*|_o + E^*_t(h_1)B_1^*|_o + E^*_t(h_2)B_2^*|_o + \widetilde\nabla_{E_t^*}B_2^*|_o \\
			& = -2B^*_1|_o - \frac{\lambda-t}{\mu}B^*_1|_o \notag \\
			& = -\frac{\lambda-t+2\mu}{\mu}X_1|_o.\notag
		\end{align}
		The others all vanish.
		It implies that the connection $ \widetilde\nabla $ defined by \eqref{eq:nabla-Lie} coincides with $ \nabla-S_\lambda(t). $
	\end{proof}	
	
	%%%-----------New subsection	
	\subsection{The case of \( - \lambda = \nu \neq \mu \)}\label{sc:space1}
	In this case, the metric $ g_{\lambda\mu\nu} $ is deformed along the spacelike Hopf fiber $ \bR\hookrightarrow \bH^3_1 \to \AdS_2. $ We now choose a direction of the Hopf fiber along $ X_1.$
	%We obtain the all non-symmetric homogeneous structures on the pseudo-Riemannian manifold $ (\bH^3_1, g_{\lambda\mu\nu}) $ as follows:
	
	\begin{theorem}\label{th:homogeneous-str-spacelike}
		If $ -\lambda = \nu \neq \mu, $ any homogeneous structure $ S $ on the Lorentzian manifold $ (\bH^3_1, g_{\lambda\mu\nu}) $ is given by the following $ S = S_\mu(t) $ or $ S_0: $ 
		\begin{equation}
			S_\mu(t) = \mu\theta^2\otimes\theta^0\wedge\theta^1 + \mu\theta^0\otimes\theta^1\wedge\theta^2 + t\theta^1\otimes\theta^2\wedge\theta^0, \label{eq:s3}
		\end{equation}
		where $ 2\nu-\mu \neq t \in \bR$ is a constant. Moreover, $ S_\mu(t) $ coincides with $ S_0 $ by substituting $ t = 2\nu-\mu.$ 
%		and $ f\in C^\infty(\bH^3_1) $ satisfies
%		$ X _0(f) = 2\cosh(f), X_1(f) = -2, X_2(f) = 2\sinh(f). $
	\end{theorem}
	\begin{proof}
		We have the following conditions from $ \widetilde\nabla R = 0 $ that the homogeneous structures $ S $ should be satisfied:
		\begin{equation}
			\rho_0 = \rho_1 = \sigma_1 = \sigma_2 = 0,\quad \sigma_0 = \rho_2 = \mu.
		\end{equation}
		Then we have
		$ S = \mu\theta^2\otimes\theta^0\wedge\theta^1 + \mu\theta^0\otimes\theta^1\wedge\theta^2 + \tau\otimes\theta^2\wedge\theta^0.$
		Therefore, from $ \widetilde\nabla S = 0, $ we obtain $ \tau_0=\tau_2=0, $ and by setting $ \tau_1 = t $ as a constant, we obtain $ S = S_\mu(t). $
	\end{proof}
	Thus, by an argument similar to that in \pref{th:isometry-timelike-1}, we obtain the following result.
	\begin{theorem}\label{th:isometry-spacelike-1}
		Assume that $ -\lambda = \nu \neq \mu. $
		If $ t \neq 2\nu - \mu, $ the isometry group $ \SU(1,1) \times \bR $ of the Lorentzian manifold $ (\bH^3_1, g_{\lambda\mu\nu}) $ acts transitively and effectively on it via the map:
		\begin{equation}\label{eq:R-action}
			\phi\colon \SU(1,1) \times \bR \times \bH^3_1 \ni \lb (A_0, s), Z \rb \mapsto A_0 Z 
			\begin{pmatrix}
				\cosh{s} & -\sinh{s} \\
				-\sinh{s} & \cosh{s}
			\end{pmatrix} 
			\in \bH^3_1.
		\end{equation}
		This leads the following reductive decomposition
		$ \bR \oplus \mathfrak{su}(1,1) = \frakh \oplus \frakm $ given by
		\begin{align}
			\frakh & = \vspan_\bR\lc 
			\begin{pmatrix}
				0 & 1 \\
				1 & 0
			\end{pmatrix} + D \rc,\\ 
			\frakm & = \vspan_\bR\lc \begin{pmatrix}
				\sqrt{-1} & 0 \\ 0 & -\sqrt{-1}
			\end{pmatrix}, \frac{\mu+t}{2\nu}\begin{pmatrix}
				0 & 1 \\ 1 & 0
			\end{pmatrix} - \frac{2\nu -\mu -t}{2\nu}D, \begin{pmatrix}
				0 & \sqrt{-1} \\ -\sqrt{-1} & 0
			\end{pmatrix} \rc,
		\end{align}
		where $ D $ is a generator of $ \bR $ that is commutative with any element in $ \mathfrak{su}(1,1), $
		induced by the homogeneous structure $ S = S_\mu(t). $
		
		If  $ t = 2\nu - \mu, \ \SU(1,1) $ transitively and effectively acts on $ \bH^3_1 $ as an isometry group from the left with the trivial reductive decomposition.
	\end{theorem}
	
	\begin{proof}
		From \pref{pr:invariant-metric}, in this case, the $ g_{\lambda\mu\nu} $ is $ G $-invariant if and only if $ G = \SU(1,1) \times \bR $ under the transitive and effective isometric action induced by $ \phi $.
		
		The non-trivial components of the connection form and the curvature tensor of the canonical connection $ \widetilde\nabla = \nabla - S_\mu(t) $ are the following:
		\begin{align}
			\omegatilde^0_2 & = -\frac{2\nu - \mu -t}{\nu}\theta^1 = \omegatilde^2_0,\label{eq:nablatilde-spacelike-1}\\
			\Rtilde(X_2,X_0) & = \frac{2(2\nu - \mu -t)}{\nu}\lb \theta^2\otimes X_0 +\theta^0\otimes X_2 \rb.\label{eq:Rtilde-spacelike}
		\end{align}
		The holonomy Lie algebra of this connection is 
		$ \widetilde\frakh = \vspan_\bR\lc \Rtilde(X_2,X_0) \rc. $
		Hence, it vanishes when $ S=S_0 $ (i.e., when $ t=2\nu - \mu). $ It implies that isotropy group of isometries is trivial.
		Hereafter, we assume that $ t \neq 2\nu - \mu, $ then it does not vanish. 
		If we set $ U = \frac{\nu}{2\nu - \mu - t}\Rtilde(X_2,X_0) = 2(\theta^2\otimes X_0 + \theta^0\otimes X_2) $ and $ S= S_\mu(t), $ the Lie algebra satisfies the following conditions: 
		\begin{align}
			[X_0,X_1] & = S_{X_0}X_1 - S_{X_1}X_0 = \frac{\mu + t}{\nu}X_2, \\
			[X_1,X_2] & =  S_{X_1}X_2 - S_{X_2}X_1 = -\frac{\mu + t}{\nu}X_0, \\
			[X_2,X_0] & =  S_{X_2}X_0 - S_{X_0}X_2 - \Rtilde(X_2,X_0) = 2X_1 -\frac{2\nu - \mu - t}{\nu}U,
		\end{align}
		\begin{equation}
			[U,X_0] = 2X_2,\quad [U,X_1] = 0,\quad [U,X_2] = 2X_1.
		\end{equation}
		Then the Lie algebra relations of $ \widetilde\frakg = \vspan_\bR\{\Uhat\}\oplus\vspan_\bR\{\Xhat_\alpha\}_{\alpha=0}^2 $ are given by 
		\begin{equation}
			[\Xhat_0, \Xhat_1]= 2\Xhat_2,\ [\Xhat_1, \Xhat_2]= -2\Xhat_0, \ [\Xhat_2, \Xhat_0] = 2\Xhat_1,\ [\Uhat, \Xhat_\alpha]=0,(\alpha=0,1,2)
		\end{equation}
		for
		\begin{align}
			\Uhat & = -\frac{\mu + t}{2\nu}U - X_1,\label{eq:hatU-R1}\\
			\Xhat_0 & = X_0 ,\label{eq:hatX0-R1}\\
			\Xhat_1 & = X_1 - \frac{2\nu - \mu - t}{2\nu}U,\label{eq:hatX1-R1}\\
			\Xhat_2 & = X_2.\label{eq:hatX2-R1}
		\end{align}
		
		On the other hand, we show that the connection $ \nabla-S_\mu(t) $ coincides with the one described $ \widetilde\nabla_{X^*}Y^*|_{o}=-[X,Y]^*|_{o} $ for $ X,Y \in \frakm, $ which is defined below, by the action $ \phi \colon \SU(1,1) \times \bR \times \bH^3_1 \to \bH^3_1. $ We also see the reductive decomposition $ \widetilde\frakg = \frakh\oplus\frakm. $
		We identify $ \{\Uhat, \Xhat_0, \Xhat_1, \Xhat_2\} $ with the basis of the Lie algebra $ \mathfrak{su}(1,1)\oplus \bR, $ that is,
		\begin{align}
			-U_1 & = -\Uhat - \Xhat_1 = -D +
			\begin{pmatrix}
				0 & -1 \\ -1 & 0
			\end{pmatrix}, \\
			B_0 & = \Xhat_0 = 
			\begin{pmatrix}
				\sqrt{-1} & 0 \\ 0 & -\sqrt{-1}
			\end{pmatrix}, \\
			B_1 & = \Xhat_1 = 
			\begin{pmatrix}
				0 & 1 \\ 1 & 0
			\end{pmatrix}, \\
			B_2 & = \Xhat_2 =
			\begin{pmatrix}
				0 & \sqrt{-1} \\ -\sqrt{-1} & 0
			\end{pmatrix}, \\
			E_t & = X_1 = \frac{\mu+t}{2\nu}\Xhat_1 + \frac{\mu+t-2\nu}{2\nu}\Uhat\\ & = 
			\begin{pmatrix}
				0 & \frac{\mu+ t}{2\nu} \\ \frac{\mu+ t}{2\nu} & 0
			\end{pmatrix} - \frac{2\nu -\mu -t}{2\nu}D.\notag
		\end{align}
		
		The reductive decomposition of the Lie algebra $ \widetilde\frakg = \frakh\oplus\frakm $ is the following:
		\begin{align*}
			\frakh & = \vspan_\bR\{U_1\} = \vspan_\bR\lc 
			\begin{pmatrix}
				0 & 1 \\ 1 & 0
			\end{pmatrix} + D \rc,\\
			\frakm & = \vspan_\bR\{ B_0,E_t,B_2\}\\
			& = \vspan_\bR\lc \begin{pmatrix}
				\sqrt{-1} & 0 \\ 0 & -\sqrt{-1}
			\end{pmatrix}, \frac{\mu+t}{2\nu}\begin{pmatrix}
				0 & 1 \\ 1 & 0
			\end{pmatrix} - \frac{2\nu -\mu -t}{2\nu}D, \begin{pmatrix}
				0 & \sqrt{-1} \\ -\sqrt{-1} & 0
			\end{pmatrix} \rc.
		\end{align*}
		We also use the Lie algebra isomorphism $ \tau \colon \frakm \ni B \mapsto B^*|_o \in T_{o}\bH^3_1 $ at $ o = (1,0) \simeq I_2 \in \SU(1,1).  $
		Then we have
		\begin{equation}
			\tau(B_\alpha) = X^*_\alpha|_{o},(\alpha=0,1,2) \quad \tau(D) = -X_1^*|_{o}, 
		\end{equation}
		\begin{equation}
			[U_1,B_0] = -2B_2, \quad [U_1, E_t] = 0, \quad [U_1, B_2] = -2B_0,
		\end{equation}
		and the remaining conditions of the Lie algebra relation are
		\begin{equation}
			[B_0, E_t] = \frac{\mu+t}{\nu}B_2, \quad [E_t,B_2] = -\frac{\mu+t}{\nu}B_0, \quad [B_2,B_0] = 2B_1.
		\end{equation}
		We have a connection $ \widetilde\nabla $ determined by the Killing vector fields generated by the element of $ \frakm $ at $ o \in \bH^3_1 $ as follows:
		\begin{align}
			\widetilde\nabla_{E_t^*}B_0^*|_{o} & = - [E_t,B_0]^*|_o = \frac{\mu + t}{\nu}B_2^*|_o = \frac{\mu + t}{\nu}X_2^*|_o ,\\
			\widetilde\nabla_{E_t^*}B_2^*|_{o} & = - [E_t,B_2]^*|_o = \frac{\mu + t}{\nu}B_0^*|_o = \frac{\mu + t}{\nu}X_0^*|_o ,\\
			\widetilde\nabla_{B_2^*}B_0^*|_{o} & = - [B_2,B_0]^*|_o = -2B_1^*|_o = -2X_1^*|_o.
		\end{align}
		We expand the left-invariant vector fields $ \{X_\alpha\}_{\alpha=0}^2 $ using the Killing vector fields generated by the elements in $ \frakm $ under the action $\phi. $
		Let $ \{f_\alpha, g_\beta,h_\gamma\}_{\alpha,\beta,\gamma=0,t,2} \subset C^\infty(\bH^3_1)$ be the coeficients as follows:
		\begin{align}
			X_0|_{(z^1,z^2)} & = f_0B^*_0|_{(z^1,z^2)}+f_tE^*_t|_{(z^1,z^2)}+f_2B^*_2|_{(z^1,z^2)},\\
			X_1|_{(z^1,z^2)} & = g_0B^*_0|_{(z^1,z^2)}+g_tE^*_t|_{(z^1,z^2)}+g_2B^*_2|_{(z^1,z^2)},\\
			X_2|_{(z^1,z^2)} & = h_0B^*_0|_{(z^1,z^2)}+h_tE^*_t|_{(z^1,z^2)}+h_2B^*_2|_{(z^1,z^2)}.
		\end{align}
		\begin{align}
			f_0(z^1,z^2) & = \frac{((z^1)^2 + (z^2)^2 + (\zbar^1)^2 + (\zbar^2)^2)(t-2\nu+\mu)-2(\mu+t)(|z^1|^2 + |z^2|^2)}{\lc(z^1)^2 - (z^2)^2 + (\zbar^1)^2 - (\zbar^2)^2\rc(t-2\nu+\mu)-2(\mu+t)},  \notag\\
			f_t(z^1,z^2) & = \frac{4\sqrt{-1}\nu(z^1z^2-\zbar^1\zbar^2)}{\lc(z^1)^2 - (z^2)^2 + (\zbar^1)^2 - (\zbar^2)^2\rc(t-2\nu+\mu)-2(\mu+t)},\label{eq:f-space}\\
			f_2(z^1,z^2) & = \frac{2\lc (\mu+t)(z^1z^2+\zbar^1\zbar^2-\zbar^1z^2-z^1\zbar^2) +2\nu(z^1\zbar^2 + \zbar^1z^2) \rc}{\lc(z^1)^2 - (z^2)^2 + (\zbar^1)^2 - (\zbar^2)^2\rc(t-2\nu+\mu)-2(\mu+t)}.\notag
		\end{align}
		\begin{align}
			g_0(z^1,z^2) & = -\frac{2\sqrt{-1}(\mu+t)(z^1\zbar^2-\zbar^1z^2)}{\lc(z^1)^2 - (z^2)^2 + (\zbar^1)^2 - (\zbar^2)^2\rc(t-2\nu+\mu)-2(\mu+t)},\notag\\
			g_t(z^1,z^2) & = -\frac{2\nu \lc (z^1)^2 - (z^2) + (\zbar^1)^2 - (\zbar^2)^2 \rc}{\lc(z^1)^2 - (z^2)^2 + (\zbar^1)^2 - (\zbar^2)^2\rc(t-2\nu+\mu)-2(\mu+t)},\label{eq:g-space}\\
			g_2(z^1,z^2) & = \frac{\sqrt{-1}(\mu+t)\lc (z^1)^2 - (z^2)^2 + (\zbar^1)^2 - (\zbar^2)^2 \rc}{\lc(z^1)^2 - (z^2)^2 + (\zbar^1)^2 - (\zbar^2)^2\rc(t-2\nu+\mu)-2(\mu+t)}.\notag
		\end{align}
		\begin{align}
			h_0(z^1,z^2) & = -\frac{2\lc(\mu+t)(z^1z^2 + \zbar^1\zbar^2 -z^1\zbar^2 - \zbar^2z^1) - 2\nu(z^1z^2+\zbar^1\zbar^2)\rc}{\lc(z^1)^2 - (z^2)^2 + (\zbar^1)^2 - (\zbar^2)^2\rc(t-2\nu+\mu)-2(\mu+t)},\notag\\
			h_t(z^1,z^2) & = -\frac{2\sqrt{-1}\nu \lc (z^1)^2 + (z^2)^2 - (\zbar^1)^2 - (\zbar^2)^2 \rc}{\lc(z^1)^2 - (z^2)^2 + (\zbar^1)^2 - (\zbar^2)^2\rc(t-2\nu+\mu)-2(\mu+t)},\label{eq:h-space}\\
			h_2(z^1,z^2) & = -\frac{2(2\nu-\mu - t)(|z^1|^2 + |z^2|^2) +(\mu + t )\lc (z^1)^2 + (z^2)^2 + (\zbar^1)^2 + (\zbar^2)^2 \rc}{\lc(z^1)^2 - (z^2)^2 + (\zbar^1)^2 - (\zbar^2)^2\rc(t-2\nu+\mu)-2(\mu+t)}.\notag
		\end{align}
		Then we calculate the connection $ \widetilde\nabla $ for the left-invariant vector fields $ \{X_\alpha\}_{\alpha=0}^2 $ at $ o\in \bH^3_1. $
		\begin{align}
			\widetilde\nabla_{X_1}X_2|_o & = E^*_t(h_0)B_0^*|_o + E^*_t(h_t)B_1^*|_o + E^*_t(h_2)B_2^*|_o + \widetilde\nabla_{E_t^*}B_2^*|_o \\
			& = -2B^*_0|_o + \frac{\mu+t}{\nu}B^*_0|_o\notag \\
			& = \frac{\mu + t-2\nu}{\nu}X_0|_o,\notag\\
			\widetilde\nabla_{X_1}X_0|_o & = E^*_t(f_0)B_0^*|_o + E^*_t(f_t)E_t^*|_o + E^*_t(f_2)B_2^*|_o + \widetilde\nabla_{E_t^*}B_2^*|_o \\
			& = -2B^*_2|_o + \frac{\mu + t}{\nu}B^*_2|_o \notag \\
			& = \frac{\mu + t-2\nu}{\nu}X_2|_o.\notag
		\end{align}
		The others all vanish. It implies that the connection $ \widetilde\nabla $ defined by \eqref{eq:nabla-Lie} coincides with $ \nabla - S_\mu(t). $
	\end{proof}

	%%---subsection$-\lambda=\mu\neq\nu$
	\subsection{The case of \(-\lambda=\mu\neq\nu\)}\label{sc:space2}
	We have left the classification of the homogeneous structures when the metric $ g_{\lambda\mu\nu} $ is deformed along the other spacelike Hopf fiber direction but these are isomorphic to the one obtained in the previous \pref{sc:space1} due to the symmetry between spacelike direction of $\AdS_3 $. If $ -\lambda = \mu \neq \nu, $
	 we obtain the following homogeneous structure $S = S_\nu(t)(\simeq S_\mu(t)) $, which coincides with $ S_0 $ by substituting $ t=2\mu-\nu. $
	\begin{equation}\label{eq:s-nu-t}
		S_\nu(t) = t\theta^2\otimes\theta^0\wedge\theta^1 + \nu\theta^0\otimes\theta^1\wedge\theta^2 + \nu\theta^1\otimes\theta^2\wedge\theta^0.(t \neq 2\mu-\nu)
	\end{equation}
	The isometry group of the Lorentzian manifold $ (\bH^3_1,g_{\lambda\mu\nu}) $ is $ \SU(1,1)\times \bR $ when $ S=S_\nu(t) $ or $\SU(1,1) $ when $ S=S_0 $ respectively.   
	An isomorphism $ \varphi $ between two homogeneous structure tensors $ \varphi^*S_\mu(t) = S_\nu(t) $ is given as follows:
	\begin{equation}\label{eq:exchange-isom}
		\varphi(X_0) = X_0, \ \varphi(X_1) = - X_2, \ \varphi(X_2) = X_1, \ \varphi(U_1) = - U_2.
	\end{equation}
	
	Using same arguments and notations in later \pref{sc:HApCM}, we have a non-trivial homogeneous almost paracontact metric structure $ (S_\nu(t), \phitilde_2,\xi_2,\eta_2 ) $ on the Lorentzian $ -\frac{\sqrt{\nu}}{\mu} $-paraSasakian manifold if $t \neq 2\mu-\nu$,
	where $ (\phitilde_2,\xi_2,\eta_2 ) = \lb \Xbar_1\otimes\thetabar^0 +  \Xbar_0\otimes\thetabar^1, \Xbar_2, \thetabar^2 \rb $ with the same group of isometries. 
	
	%---subsection$-\lambda=\mu=\nu$
	\subsection{The case of \( -\lambda=\mu=\nu \)}\label{sc:symmetric}
	In this case, the metrics of Kaluza-Klein type on $ \AdS_3 $ form a one-parameter family of Lorentzian metrics that are homothetic to the standard metric $\iota^*g_0 $ on $ \AdS_3. $ 
	Since the standard $ \AdS_3 $ has an isometry group $ \SO_0(2,2) \simeq (\SU(1,1)\times \SU(1,1))/\bZ_2, $ then the following isometric action is transitive and almost effective.
	\begin{equation}\label{eq:ful-isometry group action}
		\phi \colon\SU(1,1)_{\text{L}}\times \SU(1,1)_{\text{R}} \times \bH^3_1 \ni ((g_L,g_R), Z) \mapsto g_L Z g_R^\dagger \in \bH^3_1.
	\end{equation}
	Let $ \mathfrak{so}(2,2) = \mathfrak{su}(1,1)_{\text{L}}\oplus \mathfrak{su}(1,1)_{\text{R}} $ be the Lie algebra.
	Then $ U_0 = (X_0,X_0), \ U_1 = (X_1,-X_1), $ and $ U_2 = (X_2, -X_2) $ in $ \mathfrak{so}(2,2) $ generate a Lie subalgebra $ \frakh_o = \vspan_\bR\{U_i\}_{i=0}^2, $ which is isomorphic to $ \mathfrak{su}(1,1), $ of the stabilizer group at $ o \in \bH^3_1. $
	Then each $ B_i = (X_i, 0) $ and $ D_i = (0, X_i) $ acts on $ T_o\bH^3_1 $ from the left and right respectively, i.e.
	\begin{equation}
		B_i^*|_o = \left.\frac{d}{dt}\exp(tX_i)\cdot o\right|_{t=0} =X_i^*|_o = X_i, \ D_i^*|_o = \left.\frac{d}{dt}o\cdot\exp(tX_i^\dagger)\right|_{t=0} = {X^\dagger_i}^{\#}|_o = X^\dagger_i,
	\end{equation}
	for $ i=0,1,2. $
	We give the following lemma about the Lie algebla of this action.
	\begin{lemma}\label{lm:so(2,2)-Lie algebra}
		Assume that $ -\lambda=\mu=\nu=1. $
		Let $ \frakh \subset \frakh_0 $ be a Lie subalgebra and let $ \frakm \subset \mathfrak{so}(2,2) $ be a $ \frakh $-invariant complement.
		Then any non-trivial reductive decomposition $ \frakg = \frakm \oplus \frakh $ that is linearly isomorphic to $ \mathfrak{su}(1,1)_{\text{L}}\oplus \frakh $ is one of the following up to isomorphism in the sense of  \pref{th:isom-HS}.
		%If $ \tau\colon \frakm \ni X \mapsto \left.\frac{d }{d t}\phi_{\exp(tX)}(o)\right|_{t=0} \in T_0\bH^3_1 $ is an isomorphism, there are the following reductive decompositions $ \frakg = \frakm\oplus\frakh \subset \mathfrak{so}(2,2) $ up to isomorphism of reductive decompositions in the sense of \pref{th:isom-HS} for $ B_\pm = \frac{1}{\sqrt{2}}(B_0 \pm B_1), U_\pm = \frac{1}{\sqrt{2}}(U_0 \pm U_1), $ and $ c \in \bR. $%There are Lie algebra isomorphisms between (i) and (ii) in (1) but non of them induce isomorphism 
		\begin{enumerate}
			\item If a Lie subalgebra $ \mathfrak{su}(1,1)_{\text{L}}\oplus \bR = \frakm\oplus\frakh \subset \mathfrak{so}(2,2) $ is a reductive decomposition, it is isomorphic to one of the following:
			\begin{enumerate}
				\item $ \frakm = \vspan_\bR\{ B_0+cU_0,B_1,B_2\},\ \frakh = \vspan_\bR\{U_0\}, $
				\item $ \frakm = \vspan_\bR\{ B_0, B_1 + cU_1,B_2\},\ \frakh = \vspan_\bR\{U_1\}, $ 
				\item $ \frakm = \vspan_\bR\{ B_+, B_- + cU_+,B_2\},\ \frakh = \vspan_\bR\{U_+\}. $
%				\item $ \frakm = \vspan_\bR\{ B_+ + \frac{c}{\sqrt{2}}(U_0-U_1), B_-,B_2\},\ \frakh = \vspan_\bR\{U_0-U_1\}. $
			\end{enumerate}
			\item If a Lie subalgebra $ \frak{su}(1,1)_{\text{L}}\oplus\Aff(\bR) = \frakm \oplus \frakh \subset \mathfrak{so}(2,2) $ is a reductive decomposition, it is isomorphic to the following:
			\begin{enumerate}
				\item[(iv)] $ \frakm = \vspan_\bR\{ B_+, B_-,B_2\},\ \frakh = \vspan_\bR\{U_+, U_2\}. $
%				\item $ \frakm = \vspan_\bR\{ B_+, B_-,B_2\},\ \frakh = \vspan_\bR\{U_0-U_1, U_2\}. $
			\end{enumerate}
			\item If $ \frak{su}(1,1)_{\text{L}}\oplus\frakh_0 = \frakm\oplus\frakh_0 \simeq \mathfrak{so}(2,2) $ is a reductive decomposition, it is isomorphic to the following: 
			\begin{enumerate}
			\item[(v)]	$ \frakm = \vspan_\bR\{B_i + c U_i\}_{i=0}^2,\  \frakh = \frakh_o. $
			\end{enumerate}   
		\end{enumerate}
		Here $ B_\pm = \frac{1}{\sqrt{2}}(B_0 \pm B_1), U_\pm = \frac{1}{\sqrt{2}}(U_0 \pm U_1), $ and $ c \in \bR. $ The vector space $ \frakm $ is equipped with a Lorentzian metric via the map $ \tau\colon \frakm \to T_o\bH^3_1, $ which identifies $ (\frakm, \tau^*g_{\lambda\mu\nu}) \simeq (T_o\bH^3_1,g_{\lambda\mu\nu}). $
	\end{lemma}
	\begin{proof}
		\begin{enumerate}
			\item Assume that $ c_1c_2 \neq 0. $ When we choose a generator $ w = c_0U_0 + c_1U_1+c_2U_2 \in \frakh_o $ of $ \bR, \ $
			  $ v_1 = \frac{1}{\sqrt{c_1^2+c_2^2}}(c_2B_1 - c_1B_2), $ and $ v_2 = \frac{1}{\sqrt{c_1^2+c_2^2}}(c_1B_1 + c_2 B_2) $ satisfies the following Lie algebra relations:
			 \begin{align}
			 	[w,B_0] & = 2\sqrt{c_1^2+c_2^2}v_1,\ [w,v_1] = 2\sqrt{c_1^2+c_2^2}B_0+2c_0v_2, \ [w,v_2] = -2c_0v_1,\\
			 	[B_0,v_1] & = 2v_2,\quad [v_1,v_2] = -2B_0, \quad [v_2,B_0] = 2v_1.
			 \end{align}
			 It implies that it is enough to consider $ w = c_0U_0+c_1U_1 \in \frakh_o $ as a generator of $ \bR. $
			 Assume that $c_0c_1\neq 0.$
			 \begin{enumerate}
			 	\item When $ -c_0^2+ c_1^2 < 0, $ 
			 	we define $ v_0 $ and $ v_1 $ as follows:
			 	\begin{equation}
			 			v_0 = \frac{1}{\sqrt{c_0^2-c_1^2}}(c_1B_1+c_0B_0), \ v_1 = \frac{1}{\sqrt{c_0^2-c_1^2}}(c_0B_1+c_1B_0).
			 	\end{equation}
			 	Then we have the following Lie algebra relations:
			 	\begin{align}
			 		& [w,v_0] =0, \ [w, v_1] = 2B_2, \ [w, B_2] = -2v_1,\\
			 		& [v_0, v_1] = 2B_2, \ [v_1, B_2] = -2v_0, \ [B_2, v_0] = 2v_1.
			 	\end{align}
			 	Therefore, if $ \frakm\oplus \bR w $ is a reductive decomposition,
			 	\begin{equation}
			 		\frakm = \vspan_\bR\{ v_0 + cw, v_1,B_2 \}.
			 	\end{equation}
			 	\item When $ -c_0^2+ c_1^2 > 0, $ 
			 	we define $ v_0 $ and $ v_1 $ as follows:
			 	\begin{equation}
			 		v_0 = \frac{1}{\sqrt{c_1^2-c_0^2}}(c_0B_1+c_1B_0), \ v_1 = \frac{1}{\sqrt{c_1^2-c_0^2}}(c_1B_1+c_0B_0).
			 	\end{equation}
			 	Then we have the following Lie algebra relations:
			 	\begin{align}
			 		& [w,v_0] =-2B_2, \ [w, v_1] = 0, \ [w, B_2] = -2v_0, \\
			 		& [v_0, v_1] = 2B_2, \ [v_1, B_2] = -2v_0, \ [B_2, v_0] = 2v_1.
			 	\end{align}
			 	Therefore, if $ \frakm\oplus \bR w $ is a reductive decomposition,
			 	\begin{equation}
			 		\frakm = \vspan_\bR\{ v_0, v_1+cw,B_2 \}.
			 	\end{equation}
			 	\item When $ c_0^2=c_1^2(\neq0), $ 
			 	we define $ v_+ $ and $ v_- $ as follows:
			 	\begin{equation}
			 		v_+ = \frac{1}{\sqrt{2}}(B_0 + B_1), \ v_- = \frac{1}{\sqrt{2}}(B_0 - B_1).
  			 	\end{equation}
  			 	\begin{enumerate}
  			 		\item Suppose that $ w_+ = \frac{1}{\sqrt{2}}(U_0+U_1) \in \frakh_o $ is a generator of $ \bR. $
  			 		Then we have the following Lie algebra relations:
  			 		\begin{align}
  			 			&[w_+,v_+] =0, \ [w_+, v_-] = -2B_2, \ [w_+, B_2] = -2v_+,\\
  			 			&[v_+,v_-] = -2B_2, \ [v_+,B_2] = -2v_+, \ [v_-,B_2] = 2v_-.
  			 		\end{align}
  			 		Therefore, if $ \frakm_+\oplus \bR w_+ $ is a reductive decomposition,
  			 		\begin{equation}
  			 			\frakm_+ = \vspan_\bR\{ v_+, v_-+cw_+,B_2 \}.
  			 		\end{equation}
  			 		\item Suppose that $ w_- = \frac{1}{\sqrt{2}}(U_0-U_1) \in \frakh_o $ is a generator of $ \bR. $
  			 		Then we have the following Lie algebra relations:
  			 		\begin{align}
  			 			&[w_-,v_+] =2B_2, \ [w_-, v_-] = 0 , \ [w_-, B_2] = 2v_-,\\
  			 			&[v_+,v_-] = -2B_2, \ [v_+,B_2] = -2v_+, \ [v_-,B_2] = 2v_-. 
  			 		\end{align}
  			 		Therefore, if $ \frakm_-\oplus \bR w_- $ is a reductive decomposition,
  			 		\begin{equation}
  			 			\frakm_- = \vspan_\bR\{ v_++cw_-, v_-,B_2 \}.
  			 		\end{equation}
  			 	\end{enumerate}
  			 	%%%%%%%%%%%%%%%%%%%------------------------isomorphism
  			 	These two reductive decompositions $ \frakm_\pm \oplus \bR w_\pm $ are isomorphic via the Lie algebra isomorphism $ \varphi\colon \frakm_+\oplus \bR w_+\to \frakm_-\oplus \bR w_- $ whose restriction $ \varphi|_\frakm $ is an isometry defined as follows:
  			 	\begin{equation}
  			 		\varphi(w_+) = w_-, \ \varphi(v_+) = v_-,\ \varphi(v_- + cw_+) = v_+ + cw_-,\ \varphi(B_2) = -B_2.
  			 	\end{equation}
			 \end{enumerate}
			 Let $ \varphi \colon \frakg \to \frakg' $ be a Lie algebra isomorphism between two of the reductive decompositions in (1).
			 Since both derived algebras $ [\frakg, \frakg] $ and $ [\frakg',\frakg'] $ coincide with $ \vspan_\bR\{B_i\}_{i=0}^2, $ the restriction $ \varphi|_{\vspan_\bR\{B_i\}_{i=0}^2} $ is also a Lie algebra isomorphism. 
			 Consequently, none of the Lie algebra isomorphisms in case (1) induces an isometry of $ \frakm. $
%			 	\begin{equation}
%			 		\varphi(U_0)=c'U_1(c'\in \bR),\  \varphi([U_0,B_0]) = [\varphi(U_0), \varphi(B_0)] = 0.
%			 	\end{equation}
			 \item Assume that $ \dim \frakh = 2. $ Since $ \frakh \subset \frakh_o $ is a Lie subalgebra, one of the two generators of $ \frakh_o $ must be lightlike while the other is spacelike, and they must be orthogonal each other.
			 In this case, there are $ v_+,v_-,v_2 \in \vspan_\bR\{B_i\}_{i=0}^2 $ satisfying the following Lie algebra relation:
			 \begin{align}
			 &	[w_2, v_+] = 2v_+, [w_2, v_-] = -2v_-, [w_2,v_2] = 0,\\
			 &	[w_+, v_+] = 0, [w_+, v_-] = -2v_2, [w_+,v_2] = -2v_+,\\
			 &    [w_-, v_+] = -2v_2, [w_-, v_-] = 0, [w_-,v_2] = 2v_-.
			 \end{align}
			 \item Assume that $ \frakh = \frakh_0. $ Then $ \Ad|_H $-invariant subspace $ \frakm \subset \mathfrak{so}(2,2) $ is given by $ \frakm = \vspan_\bR\{B_i +\varphi(B_i)\}_{i=0}^2 $ for a linear map $ \varphi \colon \mathfrak{su}(1,1)\oplus\{0\} \to \frakh_0. $
			 From the condition of $ [\frakm, \frakh] \subset \frakm, $ we have the result.
		\end{enumerate}
		Thus we classified all reductive decompositions of the form $ \frakg = \mathfrak{su}(1,1)_{\text{L}}\oplus \frakh \subset \mathfrak{so}(2,2), $ where $ \frakh $ is any non-trivial Lie subalgebra of $ \frakh_o. $
	\end{proof}
	
		\begin{remark}\label{rm:Aff(R)}
$ \Aff(\bR) $ is a two-dimensional subgroup of $ \SU(1,1) $ and one of them is described as follows:
		\begin{equation}
			\Aff(\bR) = \lc \cosh v I_2 + \frac{\sinh v}{v}\begin{pmatrix}
				\frac{u}{\sqrt{2}}\sqrt{-1} & \frac{u}{\sqrt{2}} + v\sqrt{-1} \\
				\frac{u}{\sqrt{2}} - v\sqrt{-1} & -\frac{u}{\sqrt{2}}\sqrt{-1}
			\end{pmatrix} \mid u \in \bR, v \in \bR_+ \rc.
		\end{equation} 
		The Lie algebra of this representation is $ \mathfrak{aff}(\bR) = X_+^\perp = \vspan_\bR\{X_+, X_2\}, $ where $ X_+ = \frac{1}{\sqrt{2}}(X_0+X_1). $
		However, no homogeneous structures correspond to the reductive decomposition (iv) in \pref{lm:so(2,2)-Lie algebra}, since the subspace $ \frakm = \vspan_\bR\{B_0,B_1,B_2\}\subset \mathfrak{so}(2,2) $ implies that no stabilizer subalgebra is generated by $ \{\Rtilde(X_i,X_j)\}_{i,j=0}^2. $
		If an arbitrary $ c \in \bR $ in part (1), and (3) of \pref{lm:so(2,2)-Lie algebra} is zero, then the corresponding homogeneous structure become $ S_0 = \mu (\theta^0\wedge\theta^1\wedge\theta^2). $
		In that case, all reductive decompositions $ \frakg = \frakm\oplus\frakh $ are trivial, i.e. $ \frakm  =\mathfrak{su}(1,1),$ and $ \frakh = \{0\}. $
	\end{remark}
	
	Therefore, we obtain the following results summarized as in \pref{tb:AdS3-HS-Isom-RD} in  \pref{th:homogeneous-str-symmetric}.
	
	\begin{theorem}\label{th:homogeneous-str-symmetric}
		If $ -\lambda=\mu=\nu, $ every homogeneous Lorentzian structure on $ (\bH^3_1,g_{\lambda\mu\nu}) $ is one of the following up to isomorphism:
		\begin{enumerate}
			\item $ S_0 = \mu(\theta^0\wedge\theta^1\wedge\theta^2), $
			\item $ S_{\text{vol}}(t) =  t(\theta^0\wedge\theta^1\wedge\theta^2), (t \ge 0, t\neq\mu) $
			\item $ S_\lambda(t) = \mu\theta^2\otimes\theta^0\wedge\theta^1+t\theta^0\otimes\theta^1\wedge\theta^2 + \mu \theta^1\otimes\theta^2\wedge\theta^0, (t\neq \mu) $
			\item $ S_\mu(t) = \mu\theta^2\otimes\theta^0\wedge\theta^1+\mu\theta^0\otimes\theta^1\wedge\theta^2 + t \theta^1\otimes\theta^2\wedge\theta^0, (t\neq\mu) $
			\item $ S_{\text{null}}^{\mp}(t) = \mu\theta^2\otimes \theta^0\wedge\theta^1 + (\mu \pm t)\theta^0\otimes \theta^1\wedge\theta^2 + t\theta^1\otimes\theta^1\wedge\theta^2  -t \theta^0\otimes\theta^2\wedge\theta^0 + (\mu \mp t) \theta^1\otimes \theta^2\wedge\theta^0. (t> 0) $
		\end{enumerate}
		There are isomorphisms of homogeneous structures $ S_{\text{null}}^{-}(t) \simeq S_{\text{null}}^{+}(-t). $  
		In addition, the corresponding isometry groups and reductive decompositions are obtained as in the following \pref{tb:AdS3-HS-Isom-RD}.
		\begin{table}[ht]
			\centering
			\phantomsection
			\caption{All homogeneous Lorentzian structutures (HS) on $(\bH^3_1,g_{\lambda\mu\nu})$ that are homothetic to $ \AdS_3 $ with the isometry groups and reductive decompositions $ \frakg = \frakm\oplus \frakh $ in \pref{lm:so(2,2)-Lie algebra} for 
				$ t \in \bR $ and $ X_\pm = \frac{1}{\sqrt{2}}(X_0\pm X_1) $}
			\label{tb:AdS3-HS-Isom-RD}
			\begin{tabular}{|c|c|c|c|}
				\hline
				HS & isometry group & reductive decomposition & symmetry \\
				\hline\hline
				$ S_{\text{vol}}(t) $ & $ \SU(1,1)\times\SU(1,1) $ & (v) $ c = \dfrac{t-\mu}{2\mu}(\neq0) $ & Full \\
				\hline
				$ S_\lambda(t) $ & $\SU(1,1)\times U(1) $ & (i) $ c = \dfrac{t-\mu}{2\mu}(\neq0) $ & $ (X_1,X_2) $- rotations \\
				\hline
				$ S_\mu(t) $ & $\SU(1,1)\times \bR$ & (ii) $ c = \dfrac{t-\mu}{2\mu}(\neq0) $ &  $ (X_0,X_2) $-boosts \\
				%\hline
				%$ S=S_\nu(t)(\nu=\mu) $ & $\SU(1,1)\times \bR$ & $ (X_0,X_1) $-boosts \\
				\hline $ S^{\pm}_{\text{null}}(t) $ & $ \SU(1,1)\times \bR $ & (iii)  $ c = \mp\dfrac{t}{\mu}(t)(\neq0) $ & $ X_\pm $-lightlike\\
				%\hline $ S = S_0 $ & $\SU(1,1) \times \Aff(\bR)$ & (iv) & $ V^\perp $ (with $V$ null)  \\
				\hline $ S_0 $ & $\SU(1,1)$ & $ \frakm=\mathfrak{su}(1,1), \frakh = \{0\} $ & None \\
				\hline
			\end{tabular}
		\end{table}
	
	\end{theorem}
	
	\begin{proof}
		The full group of isometries of $ (\bH^3_1, g_{\lambda\mu\nu})\ (-\lambda=\mu=\nu) $ is $ \SO_0(2,2) \simeq (\SU(1,1)\times\SU(1,1))/\bZ_2 $ with the isometric action given as \eqref{eq:ful-isometry group action}.
		The images $ \{ \phi_*U_\alpha|_o \in \mathfrak{gl}(T_o\bH^3_1) \}_{\alpha=0}^2 $ of the Lie subalgebra $ \frakh_o=\vspan_\bR\{U_i\}_{i=0}^2 $ of the stabilizer group $ \Stab(o) $ are
		the following:
		\begin{equation}
			\phi_*U_\alpha|_o = 
			\begin{cases}
				& 2\lb \theta^1\otimes X_2 - \theta^2 \otimes X_1 \rb \ (\alpha=0),\\
				& -2\lb \theta^0\otimes X_2 + \theta^2 \otimes X_0 \rb \ (\alpha=1), \\
				& 2\lb \theta^0\otimes X_1 + \theta^1 \otimes X_0 \rb \ (\alpha=2).
			\end{cases}
		\end{equation}
%		\begin{align}
%			\phi_*U_0 & = 2\lb \theta^1\otimes X_2 - \theta^2 \otimes X_1 \rb,\\
%			\phi_*U_1 & = -2\lb \theta^0\otimes X_2 + \theta^2 \otimes X_0 \rb,\\
%			\phi_*U_2 & = 2\lb \theta^0\otimes X_1 + \theta^1 \otimes X_0 \rb \in \mathfrak{gl}(T_o\bH^3_1).
%		\end{align}
		Since the condition $ \widetilde\nabla S = 0 $ implies that $ X^*_\alpha(\rho_\beta) = X^*_\alpha(\sigma_\beta) = X^*_\alpha(\tau_\beta) = 0(\alpha,\beta=0,1,2), $ all coefficient functions of the homogeneous structure tensor $ S $ of the form \eqref{eq:homogeneous-str} are constants.
		Then we obtain $ \Rtilde(X_i,X_j)|_o \in \vspan_\bR\{\phi_*U_\alpha|_o\}_{\alpha=0}^2 $ for $ i,j=0,1,2 .$
		The independent and non-trivial components of the curvature tensor are the following:
		\begin{align}
			\mu^2\Rtilde^0_{101} & = 2\mu(\rho_2-\mu) - \sigma_1\tau_0 + (\mu-\sigma_0)(\mu-\tau_1),\\
			\mu^2\Rtilde^0_{102} & = -2\mu\rho_1 + \sigma_2\tau_0 +\tau_2 (\mu-\sigma_0),\\
			\mu^2\Rtilde^0_{112} & = -2\mu\rho_0 + \sigma_1\tau_2 +\sigma_2 (\mu-\tau_1),\\
			\mu^2\Rtilde^0_{201} & = -2\mu\tau_2 + \rho_0\sigma_1 +\rho_1 (\mu-\sigma_0),\\
			\mu^2\Rtilde^0_{202} & = 2\mu(\tau_1-\mu) - \rho_0\sigma_2 + (\mu-\rho_2)(\mu-\sigma_0),\\
			\mu^2\Rtilde^0_{212} & = 2\mu\tau_0 - \rho_1\sigma_2 - \sigma_1 (\mu-\rho_2),\\
			\mu^2\Rtilde^1_{201} & = -2\mu\sigma_2 + \rho_1\tau_0 +\rho_0 (\mu-\tau_1),\\
			\mu^2\Rtilde^1_{202} & = 2\mu\sigma_1 - \rho_0\tau_2 - \tau_0 (\mu-\rho_2),\\
			\mu^2\Rtilde^1_{212} & = 2\mu(\sigma_0-\mu) - \rho_1\tau_2 + (\mu-\rho_2)(\mu-\tau_1).
		\end{align}
		
		Now, we shall determine the each coefficient of the homogeneous structure tensor $ S $.
		Since $ \widetilde\nabla R = 0 $ holds trivially, the only condition we need to consider is $ \widetilde\nabla S = 0. $
		In particular, we have the following from the conditions of $ X_\alpha(\rho_\alpha) = X_\alpha(\sigma_\alpha) = X_\alpha(\tau_\alpha) = 0 \ (\alpha=0,1,2): $
		\begin{align}
			(\rho_2-\mu)(\tau_0+\sigma_1) & = 0,\ \rho_0\tau_2 + \rho_1\sigma_2 = 0, \notag\\
			(\sigma_0-\mu)(\rho_1-\tau_2) & = 0,\ \sigma_2\tau_0-\sigma_1\rho_0 = 0, \\
			(\tau_1-\mu)(\rho_0+\sigma_2) & = 0, \ \tau_0\rho_1 + \tau_2\sigma_1 = 0.\notag
		\end{align} 
		\begin{enumerate}
			\item Suppose that $ \rho_2=\sigma_0=\tau_1=\mu. $
			It immediately follows that $ S=S_0. $
			%It implies that $\rho_0=\rho_1=\sigma_1=\sigma_2=\tau_0=\tau_2 = 0 $ from $ \widetilde\nabla S =  0: $
			% hence, we have $ S = S_0. $
			\item Suppose that $ \rho_2 =\mu, \sigma_0\neq\mu, \tau_1=\mu. $ It immediately follows that $ S=S_\lambda(t). $
			\item Suppose that $ \rho_2 =\mu, \sigma_0=\mu, \tau_1\neq\mu.  $ It immediately follows that $ S=S_\mu(t). $
			\item Suppose that $ \rho_2 \neq \mu, \sigma_0=\mu, \tau_1=\mu. $ It immediately follows that $ S=S_\nu(t). $
			
			In each of the cases $ (1) $--$ (4) $ the condition $ \widetilde\nabla S = 0 $ forces the remaining coefficients to vanish. We obtained results similar to those in \pref{sc:trivial} through \pref{sc:space2} regarding the isometry groups and reductive decompositions.
			%, which are isomorphic to cases (i) and (ii) of (1) in \pref{lm:so(2,2)-Lie algebra}.
			\item Suppose that $ \rho_2=\mu, \sigma_0\neq \mu, \tau_1\neq \mu. $
			Since $ \rho_1 = \tau_2, \rho_0 = -\sigma_2, $ and $ \widetilde\nabla S = 0, $ we have the following:
			
			\begin{equation}\label{eq:constran of nablatildeS}
				\rho_0=\rho_1=\tau_2=\sigma_2=0,\ (\mu-\sigma_0)(\tau_1-\mu) + \rho_0^2 + \sigma_1\tau_0 = 0, \ \sigma_1 = -\tau_0.
			\end{equation}
			Then the non-trivial curvature tensor is described as follows:
%			\begin{align}
%				\mu^2\Rtilde^0_{101} & =\sigma_1^2+(\mu-\sigma_0)(\mu-\tau_1)=0,\\
%				\mu\Rtilde^0_{202} & =-2(\mu-\tau_1)\neq 0,\\
%				\mu\Rtilde^0_{212} & =-2\sigma_1\neq0,\\
%				\mu\Rtilde^1_{202} & =2\sigma_1\neq0,\\
%				\mu\Rtilde^1_{212} & =-2(\mu-\sigma_0)\neq 0.
%			\end{align}
%			
%			Then we have the following:
			\begin{align}
				\Rtilde(X_1, X_2) & = -\frac{2\sigma_1}{\mu}\lb \theta^2\otimes X_0 + \theta^0\otimes X_2 \rb -\frac{2(\mu-\sigma_0)}{\mu}\lb \theta^2\otimes X_1 - \theta^1\otimes X_2 \rb,\\
				\Rtilde(X_2, X_0) & = \frac{2(\mu-\tau_1)}{\mu}\lb \theta^2\otimes X_0 + \theta^0\otimes X_2 \rb -\frac{2\sigma_1}{\mu}\lb \theta^2\otimes X_1 - \theta^1\otimes X_2 \rb. 
			\end{align}
			Since $ \sigma_1^2+(\mu-\tau_1)(\mu-\sigma_0) = 0, \ \Rtilde(X_1,X_2) $ and $ \Rtilde(X_2,X_0) $ are linearly dependent. 
			Let $ k = \dfrac{\sigma_1}{\mu-\tau_1} = -\dfrac{\mu-\sigma_0}{\sigma_1} $ and $ \sigma_1 = t $ be nonzero constants.
			Every homogeneous structure can be described as the following:
			\begin{align}\label{eq:homogeneous-str-S(t,k)}
				S & = \mu\theta^2\otimes \theta^0\wedge\theta^1 + (\mu + k t)\theta^0\otimes \theta^1\wedge\theta^2 + t\theta^1\otimes\theta^1\wedge\theta^2\\
				&  -t \theta^0\otimes\theta^2\wedge\theta^0 + \lb\mu - \frac{t}{k}\rb \theta^1\otimes \theta^2\wedge\theta^0. \notag
			\end{align}
			Then we obtain the following Lie algebra relation:
			\begin{align}\label{eq:Lie algebra of reductive decomposition-k}
				[X_0, X_1] & = \lc 2 + \frac{t}{\mu}\lb k -\frac{1}{k} \rb \rc X_2, \notag\\
				[X_1, X_2] & = -2 X_0 - \frac{t}{\mu k}(kX_1 - X_0) + \frac{t}{\mu}(kU_0 - U_1),\\
				[X_2,X_0] & = 2X_1 + \frac{t}{\mu}(kX_1-X_0) -\frac{t}{\mu k}(kU_0-U_1).\notag
			\end{align}
			\begin{enumerate}
				\item When $ k^2-1 > 0, $ the following $ \frakm \oplus \frakh$ is a reductive decomposition isomorphic to (i) of (1) in \pref{lm:so(2,2)-Lie algebra} with $ c = \dfrac{t}{2\mu }\lb k - \frac{1}{k}\rb $:
				\begin{align}
					\frakm & = \vspan_\bR\lc v_0 = \frac{1}{\sqrt{k^2-1}}(k X_0 - X_1), v_1 = \frac{1}{\sqrt{k^2-1}}(kX_1 - X_0), X_2 \rc, \\
					\frakh & = \vspan_\bR\{ kU_0 - U_1\}.
				\end{align}
				\item When $ k^2-1 < 0, $ the following $ \frakm \oplus \frakh$ is a reductive decomposition isomorphic to (ii) of (1) in \pref{lm:so(2,2)-Lie algebra} with $ c = \dfrac{t}{2\mu}\lb k - \frac{1}{k}\rb $:
				\begin{align}
					\frakm & = \vspan_\bR\lc v_0 = \frac{1}{\sqrt{1- k^2}}(kX_1 - X_0), v_1 = \frac{1}{\sqrt{1 - k^2}}( kX_0 - X_1), X_2 \rc, \\
					\frakh & = \vspan_\bR\{ kU_0 - U_1\}.
				\end{align}
				\item When $ k =\mp1, $ the following $ \frakm_\mp \oplus \frakh_\mp $ is a reductive decomposition isomorphic to (iii) of (1) in \pref{lm:so(2,2)-Lie algebra} with $ c = \mp\dfrac{t}{\mu} $:
				\begin{align}
					\frakm_\mp & = \vspan_\bR\lc X_+ = \frac{1}{\sqrt{2}}(X_0 + X_1), X_- = \frac{1}{\sqrt{2}}(X_0 - X_1), X_2 \rc, \\
					\frakh_\mp & = \vspan_\bR\{ U_0 \mp U_1\}.
				\end{align}
			\end{enumerate}
			Therefore, if the causal characters of $ \Rtilde(X_1,X_2) $ and $ \Rtilde(X_2,X_0) $ are both timelike or spcelike, the Lie algebra of the corresponding reductive decomposition is isomorphic to the one obtained when $ S = S_\lambda(t) $ or $ S = S_\mu(t)(\simeq S_\nu(t)) $ from \pref{lm:so(2,2)-Lie algebra} and \pref{th:isom-HS}. 
			Then we now assume that $ \Rtilde(X_1,X_2) $ and $ \Rtilde(X_2,X_0) $ are both lightlike.
			We obtain the following homogeneous structures $ S = S^\mp_{\text{null}}(t). $
			\begin{align}
				S_{\text{null}}^{\mp}(t) & = \mu\theta^2\otimes \theta^0\wedge\theta^1 + (\mu\pm t)\theta^0\otimes \theta^1\wedge\theta^2 + t\theta^1\otimes\theta^1\wedge\theta^2\\
				&  -t \theta^0\otimes\theta^2\wedge\theta^0 + (\mu\mp t) \theta^1\otimes \theta^2\wedge\theta^0. \notag
			\end{align}
			Moreover, $ S^-_{\text{null}}(t) $ and $ S^+_{\text{null}}(-t) $ are isomorphic via the isomorphism of homogeneous structure $ \varphi \colon \frakm_-\oplus \frakh_- \to \frakm_+\oplus \frakh_+ $ given as follows:
			\begin{equation*}
				 \varphi(X_0) = X_0,\  \varphi(X_1) = -X_1,\ \varphi(X_2) = -X_2, \ \varphi(U_-) = U_+,
			\end{equation*}
			where $ \frakm_\pm=\vspan_\bR\{X_+,X_-,X_2\}, $ and $ \frakh_\pm = \vspan_\bR\{U_\pm\}. $
			\item Suppose that $ \rho_2\neq\mu, \sigma_0 = \mu, \tau_1\neq \mu. $
			Then we have $ \rho_0=\sigma_1=\sigma_2=\tau_0=0,\ \rho_1=\tau_2, $ and $ \rho_1^2-(\mu-\rho_2)(\mu-\tau_1) = 0 $ from $ \widetilde\nabla S = 0. $
			The corresponding homogeneous structures are the following :
			\begin{align}
				S & = t\theta^1\otimes \theta^0\wedge\theta^1 + \lb \mu - \frac{t}{k}\rb\theta^2\otimes \theta^0\wedge\theta^1 \\
				&  +\mu \theta^0\otimes\theta^1\wedge\theta^2 +\lb \mu - kt \rb \theta^1\otimes \theta^2\wedge\theta^0 + t \theta^2\otimes \theta^2\wedge\theta^0, \notag
			\end{align}
			where $ \rho_1=\tau_2 = t $ and $ \dfrac{t}{\mu-\rho_2} = \dfrac{\mu-\tau_1}{t} =k $ are nonzero constants. Since $  \rho_1^2-(\mu-\rho_2)(\mu-\tau_1) = 0, \ \Rtilde(X_0,X_1) $ and $ \Rtilde(X_2,X_0) $ are linearly dependent.
			Thus, the Lie algebra of the reductive decomposition is isomorphic to (ii) of (1) in \pref{lm:so(2,2)-Lie algebra}.
			It implies that the Lie algebra of the reductive decomposition is isomorphic to the one obtained when $ S = S_\mu(t)(\simeq S_\nu(t)) $ from \pref{th:isom-HS}. 
%			The curvature tensors are the following:
%			\begin{align*}
%				\mu \Rtilde(X_0,X_1) & =
%				\begin{cases}
%					& 2t\lc \lb \theta^1\otimes X_0-\theta^0\otimes X_1 \rb - \lb \theta^2\otimes X_0 - \theta^0\otimes X_2\rb \rc(S=S^+_{\rho_1=\tau_2})\\
%					& - 2t\lc \lb \theta^1\otimes X_0-\theta^0\otimes X_1 \rb + \lb \theta^2\otimes X_0 - \theta^0\otimes X_2\rb \rc(S=S^-_{\rho_1=\tau_2})
%				\end{cases}\\
%				\mu \Rtilde(X_2,X_0) & =
%				\begin{cases}
%					& 2t\lc \lb \theta^1\otimes X_0-\theta^0\otimes X_1 \rb - \lb \theta^2\otimes X_0 - \theta^0\otimes X_2\rb \rc(S=S^+_{\rho_1=\tau_2})\\
%					& 2t\lc \lb \theta^1\otimes X_0-\theta^0\otimes X_1 \rb + \lb \theta^2\otimes X_0 - \theta^0\otimes X_2\rb \rc(S=S^-_{\rho_1=\tau_2})
%				\end{cases}
%			\end{align*}
			\item Suppose that $ \rho_2\neq\mu, \sigma_0 \neq \mu, \tau_1 = \mu. $
			Then we have $ \rho_1=\sigma_1=\tau_0=\tau_2=0, \rho_0 = -\sigma_2, $ and $ \rho_0^2+(\mu-\rho_2)(\mu-\sigma_0) = 0 $ from $ \widetilde\nabla S = 0. $
			The corresponding homogeneous structures are the following:
			\begin{align}
				S & = t\theta^0\otimes \theta^0\wedge\theta^1 + \lb \mu - \frac{t}{k} \rb\theta^2\otimes \theta^0\wedge\theta^1 \\
				&  +(\mu + kt)\theta^0\otimes\theta^1\wedge\theta^2 - t \theta^2\otimes \theta^1\wedge\theta^2 + \mu \theta^1\otimes \theta^2\wedge\theta^0, \notag
			\end{align}
			where $ \rho_0=-\sigma_2 = t $ and $ \dfrac{t}{\mu-\rho_2} = -\dfrac{\mu - \sigma_0}{t} = k $ are nonzero constants.
			These homogeneous structures are isomorphic to those in (5) by the isomorphism given in \eqref{eq:exchange-isom} from \pref{th:isom-HS}.
%			The curvature tensors are the following:
%			\begin{align*}
%				\mu \Rtilde(X_0,X_1) & =
%				\begin{cases}
%					& 2t\lc \lb \theta^1\otimes X_0-\theta^0\otimes X_1 \rb + \lb \theta^2\otimes X_1 + \theta^1\otimes X_2\rb \rc(S=S^+_{\rho_0=-\sigma_2})\\
%					& - 2t\lc \lb \theta^1\otimes X_0-\theta^0\otimes X_1 \rb - \lb \theta^2\otimes X_1 + \theta^1\otimes X_2\rb \rc(S=S^-_{\rho_0=-\sigma_2})
%				\end{cases}\\
%				\mu \Rtilde(X_0,X_1) & =
%				\begin{cases}
%					& -2t\lc \lb \theta^1\otimes X_0-\theta^0\otimes X_1 \rb + \lb \theta^2\otimes X_1 + \theta^1\otimes X_2\rb \rc(S=S^+_{\rho_0=-\sigma_2})\\
%					& - 2t\lc \lb \theta^1\otimes X_0-\theta^0\otimes X_1 \rb - \lb \theta^2\otimes X_1 + \theta^1\otimes X_2\rb \rc(S=S^-_{\rho_0=-\sigma_2})
%				\end{cases}
%			\end{align*}
			\item Suppose that $ \rho_2\neq\mu, \sigma_0 \neq \mu, \tau_1\neq \mu. $
			Since it implies $ \tau_0 = -\sigma_1, \ \rho_0=-\sigma_2,\  \rho_1=\tau_2, $ we have the following conditions from $ \widetilde\nabla S = 0 $:
			\begin{align}
				\sigma_1(\mu-\rho_2)-\rho_0\rho_1 = 0, \ \rho_0^2 + \rho_1^2 + (\tau_1-\sigma_0)(\mu - \rho_2) = 0, \notag\\
				\rho_1(\mu-\sigma_0) + \rho_0\sigma_1 = 0, \ \rho_0^2 - \sigma_1^2 + (\tau_1-\rho_2)(\mu - \sigma_0) = 0,  \\
				  \rho_0(\mu-\tau_1)-\rho_1\sigma_1 = 0, \ \rho_1^2 + \sigma_1^2 + (\rho_2-\sigma_0)(\mu - \tau_1) = 0.\notag
			\end{align}
			Then $ \rho_0, \rho_1$ and $ \sigma_1 $ are either all zero or all nonzero.
			\begin{enumerate}
				\item  When $ \rho_0 \neq 0, \  \rho_1 \neq 0, \  \sigma_1 \neq 0, $ we have $ \rho_2 = \mu - \frac{\rho_0\rho_1}{\sigma_1} , \ \sigma_0 = \mu + \frac{\rho_0\sigma_1}{\rho_1} , \ \tau_1 = \mu - \frac{\rho_1\sigma_1}{\rho_0}. $ 
				Thus, we have the following curvature tensor:
				\begin{align}
					\Rtilde(X_0,X_1) &= - \frac{2\rho_0\rho_1}{\mu\sigma_1}\lb \theta^1\otimes X_0 + \theta^0\otimes X_1 \rb - \frac{2\rho_1}{\mu}\lb \theta^2\otimes X_0 + \theta^0\otimes X_2 \rb \\\
					& - \frac{2\rho_0}{\mu}\lb \theta^1\otimes X_2 - \theta^2\otimes X_1 \rb, \notag \\
					\Rtilde(X_1,X_2) &= - \frac{2\rho_0}{\mu}\lb \theta^1\otimes X_0 + \theta^0\otimes X_1 \rb - \frac{2\sigma_1}{\mu}\lb \theta^2\otimes X_0 + \theta^0\otimes X_2 \rb \\
					& - \frac{2\rho_0\sigma_1}{\mu\rho_1}\lb \theta^1\otimes X_2 - \theta^2\otimes X_1 \rb, \notag \\
					\Rtilde(X_2,X_0) & = \frac{2\rho_1}{\mu}\lb \theta^1\otimes X_0 + \theta^0\otimes X_1 \rb + \frac{2\rho_1\sigma_1}{\mu\rho_0}\lb \theta^2\otimes X_0 + \theta^0\otimes X_2 \rb \\
					& + \frac{2\sigma_1}{\mu}\lb \theta^1\otimes X_2 - \theta^2\otimes X_1 \rb. \notag
				\end{align}
				Since these generate a one-dimensional Lie subalgebra of the stabilizer group in the reductive decomposition, the following map 
				$ \varphi\colon \frak{su}(1,1)\oplus\bR \to \frak{su}(1,1)\oplus\bR $ gives an isomorphism of homogeneous structures:
				\begin{align}
					&\varphi|_\frakh = \id|_\frakh, \quad \varphi(X_0)= X_0, \\
					& \varphi(X_1) = \frac{1}{\sqrt{\sigma_1^2+\rho_0^2}}(\sigma_1 X_1 - \rho_0 X_2), \ 
					\varphi(X_2) = \frac{1}{\sqrt{\sigma_1^2+\rho_0^2}}(\rho_0 X_1 + \sigma_1 X_2). \  
				\end{align}
				 Then, the homogeneous structure is isomorphic to \eqref{eq:Lie algebra of reductive decomposition-k} when we substitute $ t = \sqrt{\sigma_1^2+\rho_0^2},$ and $  k = \dfrac{\rho_0\sigma_1}{\rho_1\sqrt{\sigma_1^2+\rho_0^2}}. $
				Therefore, the reductive decomposition coincides with case (1) in \pref{lm:so(2,2)-Lie algebra}, which implies that the homogeneous structure is isomorphic to $ S_\lambda(t), S_\mu(t) $ or $ S^\mp_{\text{null}}(t) $ from \pref{th:isom-HS}.
				\item When $\rho_0 =\rho_1=0, $ that is, $ \tau_2=\sigma_2=\rho_1=\tau_0=0, $ then we have $ S = t(\theta^0\wedge\theta^1\wedge\theta^2) $ by using the remaining condition $ X_\alpha(\rho_\beta) =  X_\alpha(\sigma_\beta) = X_\alpha(\tau_\beta) = 0.  $
				Hence, we have the following curvature tensor:
				\begin{equation}
					\Rtilde^0_{101} = \Rtilde^0_{202} = \Rtilde^1_{212} = -\frac{(\mu-t)(\mu+t)}{\mu}.
				\end{equation} 
				 When $ S=S_{\text{vol}}(0)=0, \ (\bH^3_1, g_{\lambda\mu\nu} ) $ is called a symmetric space. 
			\end{enumerate}
		\end{enumerate}
		 Then we classified not only all homogeneous structures but also their reductive decompositions and isometry groups on $ (\bH^3_1, g_{\lambda\mu\nu}) $ that is homothetic to $\AdS_3 $
		by computing the canonical connection to determine $ c $ in \pref{lm:so(2,2)-Lie algebra} summarized in \pref{tb:AdS3-HS-Isom-RD}.
	\end{proof}
	
	Therefore, the classification in \pref{tb:metric-coset} of homogeneous pseudo-Riemannian structures, including the induced isometry groups and reductive decompositions, for each metric of Kaluza-Klein type on $\AdS_3 $ has been obtained.
%	
%	\begin{theorem}
%		
%	\end{theorem}
%	\begin{proof}
%				
%	\end{proof}
		
	%%%%------------Homogeneous almost contact metric structure
	
	\section{Homogeneous almost contact metric structures on $ (\bH^3_1, g_{\lambda\mu\nu}) $ }\label{sc:HACM}
	An almost contact structure on a $ 2n+1$-dimensional manifold $ M $ is a triple $ (\phi,\xi,\eta) $ consisting of a (1,1)-type tensor field $ \phi, $ a vector field $ \xi $ called the Reeb vector field, and a one-form $ \eta $ called a contact form that satisfies the following conditions:
	\begin{equation}\label{eq:almost-contact}
		\phi^2 = -\id + \eta \otimes \xi,\quad \eta(\xi) =1.
	\end{equation}
	Then $ \phi\xi=0, \eta\circ\phi=0, $ and $ \phi $ has a rank $ 2n $. A pseudo-Riemannian metric $ g $ is compatible with an almost contact structure $ (\phi, \xi, \eta) $ if it satisfies the following condition:
	\begin{equation}\label{eq:almost-contact-metric}
		g(\phi X, \phi Y ) = g(X,Y) -\epsilon\eta(X)\eta(Y), \quad(X,Y \in \Gamma(TM))
	\end{equation}
	where $ \epsilon = g(\xi,\xi) $ is $\pm 1$ according to the causal character of the Reeb vector field $ \xi. $
	Moreover, the one-form $ \eta $ is a contact structure on $ M $ if $ \eta $ satisfies $ d\eta(X,Y) = 2g(X, \phi Y) $ for $ X,Y \in \Gamma(TM). $
	An almost contact metric structure $ (g, \phi, \xi, \eta) $ is called $ \alpha$-Sasakian structure, if it satisfies the following condition:
	\begin{equation}\label{eq:alpha-Sasaki}
		(\nabla_X \phi)Y = \alpha(g(X,Y)\xi -\epsilon\eta(Y)X). \quad (X,Y \in \Gamma(TM), \alpha \in C^\infty(M))
	\end{equation}
	If $ \alpha \equiv 1, $ we call $ (M,g, \phi, \xi, \eta) $ a Sasakian manifold.
	\begin{definition}\label{df:homogeneous-almost-contact-metric}
		A homogeneous structure $ S $ on an almost contact metric pseudo-Riemannian manifold $ (M, g, \phi,\xi,\eta) $ is called a homogeneous almost contact metric structure if it satisfies the following equations: 
		\begin{equation}\label{eq:homogeneous-almost-contact-metric}
			\widetilde\nabla \phi = 0,  \quad \widetilde\nabla \xi = 0, \quad \widetilde\nabla \eta = 0,
		\end{equation}
		where $ \widetilde\nabla = \nabla - S. $
	\end{definition}
	
	Almost contact metric structures on Kaluza-Klein type metrics on $\AdS_3$ have been studied in \cite{calvaruso2014metrics}.
	The possible choices of the Reeb vector fields $ \xi \in \Gamma(T\bH^3_1) $ satisfying $ \widetilde\nabla \xi = 0 $ for non-trivial $ S= S_\lambda(t), S_\mu(t) $ and $ S_\nu(t)$ are
	\begin{equation}
		\xi \propto
	\begin{cases}
		     & X_0(-\lambda\neq\mu=\nu),\\
		     & X_1(-\lambda=\nu\neq\mu),\\
		     & X_2(-\lambda=\mu\neq\nu).
	\end{cases} 
	\end{equation}   
	Let $ \Xbar_0 = \frac{1}{\sqrt{|\lambda|}}X_0, \Xbar_1 = \frac{1}{\sqrt{\mu}}X_1, $ and $ \Xbar_2 = \frac{1}{\sqrt{\nu}}X_2 $ be an orthonormal frame of $ g_{\lambda\mu\nu} $ and $ \thetabar^i(i=0,1,2) $ be the dual frame.
	We choose the normalized Reeb vector $ \xi $ with the same sign as $X_0, X_1 $ or $ X_2 $ and consider the following almost contact structures on $(\bH^3_1, g_{\lambda\mu\nu})$:
	\begin{align}
		(\phi_0,\xi_0,\eta_0) & = \lb \Xbar_1\otimes \thetabar^2 - \Xbar_2\otimes\thetabar^1,  \Xbar_0,  \thetabar^0\rb(-\lambda \neq \mu=\nu),\label{eq:AC-1}\\
		(\phi_1,\xi_1,\eta_1) & = \lb \Xbar_0\otimes \thetabar^2 - \Xbar_2\otimes\thetabar^0,  \Xbar_1, \thetabar^1\rb(-\lambda=\nu\neq\mu),\label{eq:AC-2}\\
		(\phi_2,\xi_2,\eta_2) & = \lb \Xbar_1\otimes \thetabar^0 - \Xbar_0\otimes\thetabar^1,  \Xbar_2, \thetabar^2\rb(-\lambda=\mu\neq\nu).\label{eq:AC-3}
	\end{align} 
Then we obtain the following from \eqref{eq:almost-contact-metric}.

\begin{lemma}\label{lm:compatible-condition}
	The almost contact structures described in \eqref{eq:AC-1}, \eqref{eq:AC-2} and \eqref{eq:AC-3} are compatible with the Kaluza-Klein type metric $g_{\lambda\mu\nu}$ if and only if $ \sgn(\mu)= \sgn(\nu) $ for $ (\phi,\xi,\eta) = (\phi_0,\xi_0,\eta_0),  \sgn(-\lambda) \neq \sgn(\nu) $ for $ (\phi,\xi,\eta) = (\phi_1,\xi_1,\eta_1) $ and $ \sgn(-\lambda) \neq \sgn(\mu) $ for $ (\phi,\xi,\eta) = (\phi_2,\xi_2,\eta_2). $
\end{lemma} 
	\begin{theorem}\cite{calvaruso2014metrics}\label{th:trivial-HACM}
		Every left-invariant almost contact metric structure is a homogeneous almost contact metric structure induced by the canonical connection $ \widetilde\nabla = \nabla -S_0 $ with the isometry group $\SU(1,1) $ and the trivial reductive decomposition.
		%Moreover, we have a necessary and sufficient condition of being contact metric structure for each $(\phi_l,\xi_l,\eta_l)_{l=0,1,2} $ in \pref{tb:compatible-ACS}.
		Moreover, $(\bH^3_1,g_{\lambda\mu\nu},\phi_l,\xi_l,\eta_l)_{l=1,2} $ are contact metric manifolds if and only if $ |\lambda| = \mu\nu $ for $ l=0 $ and $\mu = \lambda\nu $ for $ l =1. $
		Then the following \pref{tb:compatible-ACS} is obtained.
		
		\begin{table}[ht]
			\caption{The conditions of each almost contact structure being compatible and contact metric}
			\label{tb:compatible-ACS}
			\centering
			\begin{tabular}{|c|c|c|}
				\hline
				almost contact structure & compatible with the metric & contact metric \\
				\hline\hline
				$ (\phi_0,\xi_0,\eta_0) $ & Any & $ |\lambda|=\mu\nu $ \\
				\hline
				$ (\phi_1,\xi_1,\eta_1) $ & $\lambda>0$(Riemannian) & $ \mu = \lambda\nu $\\
				\hline
				$ (\phi_2,\xi_2,\eta_2) $ & $\lambda>0$(Riemannian) & $ \nu = \lambda\mu $ \\
				\hline
			\end{tabular}
		\end{table}
	\end{theorem}
	Taking into acount \ref{th:trivial-HACM}, we focus on investigating homogenoeus contact metric structures with non-trivial homogeneous structures $ S \neq S_0. $
	%Taking into account \pref{th:trivial-HACM}, we investigate non-trivial homogeneous almost contact metric structures on $ (\bH^3_1, g_{\lambda\mu\nu}) $ further without considering the case of $ -\lambda\neq\mu,\mu\neq\nu, \nu \neq -\lambda. $ 

	%%%----subsection
	\subsection{The case of \( - \lambda \neq \mu = \nu \)}
	In this case, the almost contact structure $ (\phi,\xi,\eta) = (\phi_0,\xi_0,\eta_0)$ is compatible with $g_{\lambda\mu\nu} $ and satisfied with $ (\nabla-S_\lambda(t))\xi_0 = 0.$ 
	Since $(\bH^3_1,g_{\lambda\mu\nu},\phi_0,\xi_0,\eta_0) $ is an almost contact metric manifold, we compute $ (\nabla_{X_\alpha}\phi)X_\beta(\alpha,\beta=0,1,2) $ using $ (\nabla_X \phi)Y = \nabla_X(\phi Y) - \phi(\nabla_XY) $ as follows: 
	\begin{align}
		(\nabla_{X_0}\phi)X_0 & = 0, \ \qquad (\nabla_{X_0}\phi)X_1 = 0, \quad (\nabla_{X_0}\phi)X_2 = 0,\\
		(\nabla_{X_1}\phi)X_0 & = -\frac{\lambda}{\mu}X_1, (\nabla_{X_1}\phi)X_1 = X_0, \ (\nabla_{X_1}\phi)X_2 = 0,\\
		(\nabla_{X_2}\phi)X_0 & = -\frac{\lambda}{\mu}X_2, (\nabla_{X_2}\phi)X_1 = 0, \quad (\nabla_{X_2}\phi)X_2 = X_0.
	\end{align}
	The non-trivial components of $ g_{\lambda\mu\nu}(X,Y)\xi -\epsilon \eta(Y)X $ are given by
	\begin{align}
		g_{\lambda\mu\nu}(X_1,X_0)\xi -\epsilon \eta(X_0)X_1 & = -\epsilon\sqrt{|\lambda|}X_1, \\
		g_{\lambda\mu\nu}(X_1,X_1)\xi -\epsilon \eta(X_1)X_1 &= \frac{\mu}{\sqrt{|\lambda|}}X_0, \\
		g_{\lambda\mu\nu}(X_2,X_0)\xi -\epsilon \eta(X_0)X_2 & = -\epsilon\sqrt{|\lambda|}X_2,\\
		g_{\lambda\mu\nu}(X_2,X_2)\xi -\epsilon \eta(X_2)X_2 & = \frac{\mu}{\sqrt{|\lambda|}}X_0.
	\end{align}
	In comparison with \eqref{eq:alpha-Sasaki}, $ (\bH^3_1, g_{\lambda\mu\nu}, \phi_0,\xi_0,\eta_0) $ is an $ \alpha $-Sasakian manifold with $ \alpha = \frac{\sqrt{|\lambda|}}{\mu}. $
	In this case, we obtain the following theorem since $ \widetilde\nabla \phi_0 = 0 $ and $ \widetilde\nabla \eta_0 = 0 $ hold for $\widetilde\nabla =\nabla -S_\lambda(t)$.
	
	\begin{theorem}\label{th:timelike-homogeneous-almost-contact}
		Assume that $  -\lambda \neq  \mu = \nu. $ The quadruple $ (S_\lambda(t),\phi_0,\xi_0,\eta_0) $  is a homogeneous almost contact metric structure on the pseudo-Riemannian $ \frac{\sqrt{|\lambda|}}{\mu} $-Sasakian manifold $ (\bH^3_1, g_{\lambda\mu\nu}) $ with the same isometry group and the reductive decomposition in \pref{th:isometry-timelike-1}.
	\end{theorem}
	\begin{remark}
		In this case, the almost contact metric structures are contact metric structures if and only if $\mu = \sqrt{|\lambda|}. $
		Therefore, it is equivalent to the condition of being Sasaki.
	\end{remark}
	
	%%--subsection
	\subsection{The case of \( -\lambda = \nu \neq \mu \)}
	In this case, there are no almost contact metric structures on $ (\bH^3_1, g_{\lambda\mu\nu}) $ from \pref{th:trivial-HACM}. We have the following:

	\begin{theorem}\label{th:spacelike-homogeneous-almost-contact}
		Assume that $ - \lambda = \nu \neq \mu.$
		There are no homogeneous almost contact metric structures on the Lorentzian manifold $ (\bH^3_1, g_{\lambda\mu\nu}). $    	
	\end{theorem}

\subsection{The case of \( -\lambda=\mu=\nu \)}
If $ -\lambda=\mu=\nu, $ the almost contact metric structure $ (\phi_0, \xi_0, \eta_0) $ is a Lorentzian $\frac{1}{\sqrt{-\lambda}}$-Sasakian structure on $ (\bH^3_1, g_{\lambda\mu\nu}). $ 
The condition of $\widetilde\nabla \phi = 0 $ implies that $S = -\lambda\theta^2\otimes\theta^0\wedge\theta^1 + \sigma\otimes\theta^1\wedge\theta^2-\lambda\theta^1\otimes\theta^2\wedge\theta^0. $ By an argument analogous to that in \pref{th:homogeneous-str-timelike}, we have the following result. 

\begin{theorem}
	If $ -\lambda = \mu =\nu,  $ the quadruple $ (S_\lambda(t),\phi_0, \xi_0, \eta_0) $ is a homogeneous almost contact metric structure on the Lorentzian $ \frac{1}{\sqrt{-\lambda}} $-Sasakian manifold $ (\bH^3_1, g_{\lambda\mu\nu}, \phi_0, \xi_0, \eta_0). $
\end{theorem}

We obtain the following theorem by substituting $-\lambda=\mu=\nu $ in the proof of \pref{th:isometry-timelike-1}.

\begin{theorem}\label{th:homogeneous-contact-standard-1}
	Assume that $ -\lambda=\mu=\nu. $ If $ t\neq-\lambda, $ the isometry group $ \SU(1,1) \times U(1) $ acts transitively and almost effectively on the Lorentzian $\frac{1}{\sqrt{-\lambda}}$-Sasakian manifold $(\bH^3_1,g_{\lambda\mu\nu}) $ as in \pref{th:isometry-timelike-1} induced by the homogeneous almost contact metric structure $ (S_\lambda(t),\phi_0,\xi_0,\eta_0) $ with the following reductive decomposition $\frakg = \frakm\oplus \frakh:$
	\begin{align}
		\frakh & = \vspan_\bR\lc \begin{pmatrix}
			0 & 0 \\ 0 & \sqrt{-1}
		\end{pmatrix}\rc,\\
		\frakm & = \vspan_\bR\lc \begin{pmatrix}
			\sqrt{-1} & 0 \\ 0 & -\sqrt{-1}\frac{t}{-\lambda}
		\end{pmatrix}, \begin{pmatrix}
			0 & 1 \\ 1 & 0
		\end{pmatrix}, 
		\begin{pmatrix}
			0 & \sqrt{-1} \\ -\sqrt{-1} & 0
		\end{pmatrix}
		\rc.
	\end{align}
	If $ t = -\lambda, $ the isometry group is $ \SU(1,1) $
	with the trivial reductive decomposition.
\end{theorem}

In this case, $ (S_\lambda(t),\phi_0, \xi_0, \eta_0) $  is the unique non-trivial homogeneous almost contact metric structure up to isomorphism.
	
	%----------almost paracontact
	
	\section{Homogeneous almost paracontact metric structures on $ (\bH^3_1, g_{\lambda\mu\nu}) $ }\label{sc:HApCM}
	An almost paracontact structure on a $ 2n+1$-dimensional manifold $ M $ is a triple $ (\phi,\xi,\eta) $ consisting of a (1,1)-type tensor field $ \phi, $ a vector field $ \xi $, and a one-form $ \eta $ that satisfies the following conditions:
	\begin{equation}\label{eq:almost-paracontact}
		\phi^2 = \id - \eta \otimes \xi,\quad \eta(\xi) =1.
	\end{equation}
	Then $ \phi\xi=0, \eta\circ\phi=0 $ and $ \phi $ has a rank $ 2n $ as well as an almost contact structure. A pseudo-Riemannian metric $ g $ is compatible with an almost paracontact structure $ (\phi, \xi, \eta) $ if it satisfies the following condition:
	\begin{equation}\label{eq:almost-paracontact-metric}
		g(\phi X, \phi Y ) = -g(X,Y) + \eta(X)\eta(Y).\quad(X,Y \in \Gamma(TM))
	\end{equation}
	Any pseudo-Riemannian metric $ g $ compatible with a given paracontact structure is necessarily of signature $ (n,n+1) $; hence the Reeb vector field $\xi $ must be spacelike.
	The one-form $ \eta $ is called a paracontact structure on $ M $ if $ \eta $ satisfies $ d\eta(X,Y) = 2g(X, \phi Y) $ for $ X,Y \in \Gamma(TM). $
	An almost paracontact metric structure $ (g, \phi, \xi, \eta) $ is called a $ \beta$-paraSasakian structure, if it satisfies the following condition:
	\begin{equation}\label{eq:beta-paraSasaki}
		(\nabla_X \phi)Y = \beta(g(X,Y)\xi -\eta(Y)X). \quad (X,Y \in \Gamma(TM), \beta \in C^\infty(M))
	\end{equation}
	If $ \beta \equiv -1, $ we call $ (M,g, \phi, \xi, \eta) $ a paraSasakian manifold.
	\begin{definition}\label{df:homogeneous-almost-paracontact-metric}
		A homogeneous structure $ S $ on an almost paracontact metric pseudo-Riemannian manifold $ (M, g, \phi,\xi,\eta) $ is called a homogeneous almost paracontact metric structure if it satisfies the following equations: 
		\begin{equation}\label{eq:homogeneous-almost-paracontact-metric}
			\widetilde\nabla \phi = 0,  \quad \widetilde\nabla \xi = 0, \quad \widetilde\nabla \eta = 0, 
		\end{equation}
		where $ \widetilde\nabla = \nabla - S. $
	\end{definition}
	Almost paracontact metric structures on Kaluza-Klein type metrics on $\AdS_3$ have also been studied in \cite{calvaruso2014metrics}.
	Hereafter, we assume that $ \lambda<0 $ since $ (\bH^3_1, g_{\lambda\mu\nu}) $ with almost paracontact metric structures must be Lorentzian.
	
	The following almost paracontact structures in \eqref{eq:ApC-2} and \eqref{eq:ApC-3} are compatible with the Kaluza-Klein type metric $g_{\lambda\mu\nu}$ on $ \bH^3_1 $.
	\begin{align}
		%(\phitilde_0,\xi_0,\eta_0) & = \lb -\Xbar_1\otimes \thetabar^2 - \Xbar_2\otimes\thetabar^1,  \Xbar_0,  \thetabar^0\rb(-\lambda \neq \mu=\nu), \label{eq:ApC-1}\\
		(\phitilde_1,\xi_1,\eta_1) & = \lb -\Xbar_2\otimes \thetabar^0 - \Xbar_0\otimes\thetabar^2,  \Xbar_1, \thetabar^1\rb,\label{eq:ApC-2}\\
		(\phitilde_2,\xi_2,\eta_2) & = \lb \Xbar_1\otimes \thetabar^0  + \Xbar_0\otimes\thetabar^1,  \Xbar_2, \thetabar^2\rb.\label{eq:ApC-3}
	\end{align} 
	By a similar argument as for almost contact structures, the following hold.
	\begin{theorem}\cite{calvaruso2014metrics}\label{th:trivial-HApCM}
		Every left-invariant  almost paracontact metric structure is a homogeneous almost paracontact metric structure induced by the canonical connection $ \widetilde\nabla = \nabla -S_0 $ with the isometry group $\SU(1,1) $ and the trivial reductive decomposition.
		Moreover, $(\bH^3_1,g_{\lambda\mu\nu},\phitilde_l,\xi_l,\eta_l)_{l=1,2} $ are paracontact metric manifolds if and only if $\mu = -\lambda\nu $ for $ l=1 $ and $\nu =-\lambda\mu $ for $ l =2. $
		Then the following \pref{tb:compatible-ApCS} is obtained.
		\begin{table}[ht]
			\caption{The conditions of each almost paracontact structure being paracontact metric }
			\label{tb:compatible-ApCS}
			\centering
			\begin{tabular}{|c|c|}
				\hline
				almost paracontact structure & paracontact metric \\
				\hline
				$ (\phitilde_1,\xi_1,\eta_1) $ & $ \mu =-\lambda\nu $\\
				\hline
				$ (\phitilde_2,\xi_2,\eta_2) $ & $ \nu = -\lambda\mu $\\
				\hline
			\end{tabular}
		\end{table}
	\end{theorem}
	
	Taking into account \pref{th:trivial-HApCM}, we focus on investigating homogeneous almost paracontact metric structures with non-trivial homogeneous structures $ S \neq S_0 $ on $ (\bH^3_1, g_{\lambda\mu\nu}). $

	%%--subsection
	\subsection{The case of \(-\lambda\neq\mu=\nu \)}
	In this case,  we have the following:
	
	\begin{theorem}\label{th:timelike-homogeneous-almost-paracontact}
		Assume that $ - \lambda \neq \mu = \nu.$
		There are no non-trivial homogeneous almost paracontact metric structures on the Lorentzian manifold $ (\bH^3_1, g_{\lambda\mu\nu}). $  
	\end{theorem}
	\begin{proof}
		In this case, the condition $ \widetilde\nabla \xi =0 $ implies $ \xi $ is proportional to $ X_0 $ since $ S = S_\lambda(t). $ However, a timelike vector cannot be the Reeb vector field of any almost paracontact metric structure. 
	\end{proof}
	%new subsection
	\subsection{The case of \(-\lambda=\nu\neq\mu\)}
	In this case, the almost paracontact structure $ (\phi,\xi,\eta) = (\phitilde_1,\xi_1,\eta_1)$ is compatible with $g_{\lambda\mu\nu} $ and satisfied with $ (\nabla-S_\mu(t))\xi_1 = 0.$ 
	Since $(\bH^3_1,g_{\lambda\mu\nu},\phitilde_1,\xi_1,\eta_1) $ is an almost paracontact metric manifold, we compute $ (\nabla_{X_\alpha}\phi)X_\beta(\alpha,\beta=0,1,2) $ using $ (\nabla_X \phi)Y = \nabla_X(\phi Y) - \phi(\nabla_XY) $ as follows: 
	\begin{align}
		(\nabla_{X_0}\phi)X_0 & = X_1, \quad (\nabla_{X_0}\phi)X_1 = \frac{\mu}{\nu}X_0, \ (\nabla_{X_0}\phi)X_2 = 0,\\
		(\nabla_{X_1}\phi)X_0 & = 0, \qquad (\nabla_{X_1}\phi)X_1 = 0, \qquad (\nabla_{X_1}\phi)X_2 = 0,\\
		(\nabla_{X_2}\phi)X_0 & = 0, \qquad (\nabla_{X_2}\phi)X_1 = \frac{\mu}{\nu}X_2, \ (\nabla_{X_2}\phi)X_2 = -X_1.
	\end{align}
	The non-trivial components of $ g_{\lambda\mu\nu}(X,Y)\xi - \eta(Y)X $ are given by
	\begin{align}
		g_{\lambda\mu\nu}(X_0,X_0)\xi - \eta(X_0)X_0 & = -
		\frac{\nu}{\sqrt{\mu}}X_1, \\
		g_{\lambda\mu\nu}(X_0,X_1)\xi - \eta(X_1)X_0 &= -\sqrt{\mu}X_0, \\
		g_{\lambda\mu\nu}(X_2,X_1)\xi - \eta(X_1)X_2 & = -\sqrt{\mu}X_2,\\
		g_{\lambda\mu\nu}(X_2,X_2)\xi - \eta(X_2)X_2 & = \frac{\nu}{\sqrt{\mu}}X_1.
	\end{align}
	In comparison with \eqref{eq:beta-paraSasaki}, $ (\bH^3_1, g_{\lambda\mu\nu}, \phitilde_1,\xi_1,\eta_1) $ is a $ \beta $-paraSasakian manifold with $ \beta = -\frac{\sqrt{\mu}}{\nu}. $
	Then we obtain the following theorem since $ \widetilde\nabla \phitilde_1 = 0 $ and $ \widetilde\nabla \eta_1 = 0 $ hold for $\widetilde\nabla = \nabla - S_\mu(t)$.
	
	\begin{theorem}\label{th:spacelike-homogeneous-almost-paracontact}
		Assume that $  -\lambda =  \nu \neq \mu. $ The quadruple $ (S_\mu(t),\phitilde_1,\xi_1,\eta_1) $  is a homogeneous almost paracontact metric structure on the Lorentzian $ -\frac{\sqrt{\mu}}{\nu} $-paraSasakian manifold $ (\bH^3_1, g_{\lambda\mu\nu}) $ with the same isometry group and the reductive decomposition in \pref{th:isometry-spacelike-1}.
	\end{theorem}
	\begin{remark}
		In this case, the almost paracontact metric structures are paracontact metric structures if and only if $\nu = \sqrt{\mu}. $
		Therefore, it is equivalent to the condition of being paraSasaki.
	\end{remark}
	
	%---subsection-standard
	\subsection{The case of \( -\lambda=\mu=\nu \)}
	If $ -\lambda=\mu=\nu, $ the almost paracontact metric structure $(\phitilde_1,\xi_1,\eta_1) $ is a Lorentzian $ -\frac{1}{\sqrt{\mu}} $-paraSasakian structure on $ (\bH^3_1, g_{\lambda\mu\nu}).$
	The condition of $\widetilde\nabla \phi = 0 $ implies that $S = \mu\theta^2\otimes\theta^0\wedge\theta^1 + \mu\theta^0\otimes\theta^1\wedge\theta^2+\tau\otimes\theta^2\wedge\theta^0. $ By an argument analogous to that in \pref{th:homogeneous-str-spacelike}, we obtain the following result.
	
	\begin{theorem}
		If $ -\lambda = \mu =\nu,  $ the quadruple $(S_\mu(t),\phitilde_1, \xi_1, \eta_1) $ is a homogeneous almost paracontact metric structure on the Lorentzian $ -\frac{1}{\sqrt{\mu}} $-paraSasakian manifold $ (\bH^3_1, g_{\lambda\mu\nu}, \phitilde_1, \xi_1, \eta_1). $
	\end{theorem}
	
	We obtain the following theorem by substituting $-\lambda=\mu=\nu $ in the proof of \pref{th:isometry-spacelike-1}.
	
	\begin{theorem}\label{th:homogeneous-paracontact-standard}
		Assume that $ -\lambda=\mu=\nu. $ If $ t\neq\mu, $ the isometry group $ \SU(1,1) \times \bR $ acts transitively and effectively on the Lorentzian $ -\frac{1}{\sqrt{\mu}} $-paraSasakian manifold $(\bH^3_1,g_{\lambda\mu\nu}) $ as in \pref{th:isometry-spacelike-1} induced by the homogeneous almost paracontact metric structure  $ (S_\mu(t),\phitilde_1,\xi_1,\eta_1) $ with the following reductive decomposition $\frakg = \frakm\oplus \frakh:$
		\begin{align}
			\frakh & = \vspan_\bR\lc 
			\begin{pmatrix}
				0 & 1 \\
				1 & 0
			\end{pmatrix} + D \rc,\\ 
			\frakm & = \vspan_\bR\lc \begin{pmatrix}
				\sqrt{-1} & 0 \\ 0 & -\sqrt{-1}
			\end{pmatrix}, \frac{\mu+t}{2\mu}\begin{pmatrix}
				0 & 1 \\ 1 & 0
			\end{pmatrix} - \frac{\mu -t}{2\mu}D, \begin{pmatrix}
				0 & \sqrt{-1} \\ -\sqrt{-1} & 0
			\end{pmatrix} \rc,
		\end{align}
		If $ t = \mu, $ the isometry group is $ \SU(1,1) $
		with the trivial reductive decomposition.
	\end{theorem}
	
	In this case, $ (S_\mu(t),\phitilde_1,\xi_1,\eta_1) $ is the unique non-trivial homogeneous almost paracontact metric structure.
	
	%%%%%-------Homogeneous metric mixed three structures
	\section{Homogeneous mixed metric 3-structures on \((\bH^3_1,g_{\lambda\mu\nu})\)}\label{sc:mixed metric 3-structure}
	Mixed metric 3-structures, which were introduced in \cite{ianucs2006real} as the odd-dimensional counterpart of a paraquaternionic structure, associated with the Kaluza-Klein type metric on $\AdS_3 $ have also been studied in \cite{calvaruso2014metrics}.
	
	A mixed 3-structure on a manifold is a triple of structures $ (\phi_l,\xi_l,\eta_l)_{l=0,1,2}, $ consisting of almost contact and paracontact structures satisfying
	\begin{equation}
		\begin{cases}
		& \eta_i(\xi_j) = 0,\\
		& \eta_i\phi_j = \tau_k\eta_k = -\eta_j\phi_i,\\
		& \phi_i(\xi_j) = \tau_j\xi_k, \quad \phi_j(\xi_i) =-\tau_i\xi_k,\\
		& \phi_i\phi_j-\tau_i\eta_j\otimes\xi_i = \tau_k\phi_k = -\phi_j\phi_i + \tau_j\eta_i\otimes\xi_j, 
	\end{cases}
	\end{equation}
	for every cyclic permutation of the indices $i,j,k=0,1,2, $ where $ \tau_l = 1 $ if $ ( \phi_l,\xi_l,\eta_l) $ is almost contact and $ \tau_l=-1 $
	if $ ( \phi_l,\xi_l,\eta_l) $ is almost paracontact. 
	If in addition $ M $ admits a pseudo-Riemannian metric $ g $ such that 
	\begin{equation}\label{eq:mixed metric 3-structure}
		g(\phi_lX,\phi_lY) = \tau_l\lb g(X,Y) - \epsilon_l\eta_l(X)\eta_l(Y) \rb, \ (X,Y\in \Gamma(TM), l=0,1,2)
	\end{equation}
	where $\epsilon_l = g(\xi_l,\xi_l), $ then $ g $ is compatible with a mixed 3-structure $(\phi_l,\xi_l,\eta_l)_{l=0,1,2}. $
	In the special case \pref{th:mixed-Lorentz} where the structure is Sasakian for $ l= 0 $ and paraSasakian for $l=1,2$, the mixed metric 3-structure is called a mixed 3-Sasakian structure.

	A mixed metric 3-structure $(\phi_l,\xi_l,\eta_l)_{l=0,1,2} $ is called homogeneous if there is a homogeneous structure $ S $ of the canonical connection $\widetilde\nabla = \nabla-S $ satisfying the parallel condition of $(\phi_l,\xi_l,\eta_l)_{l=0,1,2}. $
	Obviously, there are no groups of isometries with non-trivial isotropic subgroups that admit homogeneous mixed metric 3-structures, we retrieve the following.
	 \begin{theorem}\cite{calvaruso2014metrics}\label{th:mixed-Lorentz}
	 For an arbitrary Lorentzian metric of Kaluza-Klein type on $\AdS_3, $ there is the following homogeneous mixed metric 3-structures for the canonical connection $\widetilde\nabla = \nabla - S_0 $:
	 \begin{equation}\label{eq:mixed-Lorentz}
	 	(\phi_l,\xi_l,\eta_l)_{i=0,1,2} = 
	   \begin{cases}
	 		(\phi_0,\xi_0,\eta_0) & = \lb \Xbar_1\otimes\thetabar^2-\Xbar_2\otimes\thetabar^1, \Xbar_0,\thetabar^0 \rb \\
	 		(\phi_1,\xi_1,\eta_1) & =  \lb -\Xbar_2\otimes\thetabar^0-\Xbar_0\otimes\thetabar^2, \Xbar_1,\thetabar^1 \rb \\
	 		(\phi_2,\xi_2,\eta_2) & = \lb \Xbar_1\otimes\thetabar^0 + \Xbar_0\otimes\thetabar^1, \Xbar_2,\thetabar^2 \rb 
	 	\end{cases},
	 \end{equation}
	 induced by the isometry group $ \SU(1,1) $ with the trivial reductive decomposition.
	 Moreover, if $-\lambda=\mu=\nu=1, $ the mixed metric 3-structure is a mixed 3-Sasakian structure with the isometry group $ \SU(1,1) $.  
	 \end{theorem}
	 
	 \begin{remark}
	 	If $ -\lambda=\mu=\nu, $ then the mixed 3-Sasakian structure satisfies $ \phi_0 = \frac{1}{2}U_0, \phi_1 = \frac{1}{2}U_1 $ and $ \phi_2 = \frac{1}{2}U_2. $
	 \end{remark}
	 
	 	 \begin{theorem}\cite{calvaruso2014metrics}\label{th;mixed-Riemann}
	 	For an arbitrary Riemannian metric of Kaluza-Klein type on $\AdS_3, $ there is the following homogeneous mixed metric 3-structures for the canonical connection $\widetilde\nabla = \nabla - S_0 $:
	 	\begin{equation}\label{eq:mixed-Riemann}
	 		(\phi_l,\xi_l,\eta_l)_{i=0,1,2} = 
	 		\begin{cases}
	 			(\phi_0,\xi_0,\eta_0) & = \lb -\Xbar_1\otimes \thetabar^2 + \Xbar_2\otimes\thetabar^1,  \Xbar_0,  \thetabar^0\rb\\
	 			(\phi_1,\xi_1,\eta_1) & = \lb \Xbar_0\otimes \thetabar^2 - \Xbar_2\otimes\thetabar^0,  \Xbar_1, \thetabar^1\rb\\
	 			(\phi_2,\xi_2,\eta_2) & = \lb \Xbar_1\otimes \thetabar^0 - \Xbar_0\otimes\thetabar^1,  \Xbar_2, \thetabar^2\rb
	 		\end{cases},
	 	\end{equation}

	 	induced by the isometry group $ \SU(1,1) $ with the trivial reductive decomposition.
	 	Moreover, these mixed metric 3-structures are never 3-Sasakian. 
	 \end{theorem}
	 
	  %%%%------concluding remarks and comments

	%%%%%%-----Acknowledgement
	\section*{Acknowledgement}
	I would like to thank Satsuki Matsuno for helpful discussions and comments.

		%Appendix
	\appendix
	\renewcommand{\thetheorem}{\Alph{section}.\arabic{theorem}}
	\renewcommand{\theequation}{\Alph{section}.\arabic{equation}}
	
	\section{Deformation of the metrics and the Hopf fibrations}\label{sc:Hopf-fibration}
	We call any principal $ \bS^1 $- or $ \bR $-bundle over $ \AdS_3 $ whose fibers are the orbits of left (resp. right)-invariant vector fields, which are Killing vector fields of the right (resp. left) action, a Hopf fibration.
	In this section, we shall see that Hopf fibrations, which are associated with each homogeneous structure, on $\AdS_3 $ provide examaples of two types of Killing submersions by left-invariant vector fields. 
	There are three kinds of Killing vector fields on a Lorentzian manifold $ (\bH^3_1, g_{\lambda\mu\nu}) $ generated by invariant vector fields. The timelike and spacelike left-invariant vectors can be regarded as deforming the metric $ g_{\lambda\mu\nu} $ away from the metric that is homothetic to the standard $ \AdS_3 $ (i.e.,  when $ -\lambda=\mu=\nu). $ These are classified by the Hopf-fibration in the following \pref{tb:deform-Hopf-fibration}.
	
		\begin{table}[ht]
		\caption{The deformation of the metric homothetic to the standard $\AdS_3 $ along the fiber of the Hopf fibration and associated homogeneous pseudo-Riemannian structures (HS)}
		\label{tb:deform-Hopf-fibration}
		\centering
		\begin{tabular}{|c|c|c|c|}
			\hline
			Hopf-fibration & fiber direction & metric & HS \\
			\hline
			$ \bS^1\hookrightarrow \bH^3_1 \to \bH^2(4) $ & $ X_0 $ & $ \diag(\lambda,\mu,\mu) $ & $ S_\lambda(t) $  \\
			\hline
			$ \bR \hookrightarrow \bH^3_1 \to \bH^2_1(4) $ & $ X_1 $ & $ \diag(-\nu,\mu,\nu) $ & $ S_\mu(t)(\simeq S_\nu(t)) $  \\
			\hline
			$  \bR_+\hookrightarrow\bH^3_1 \to \bS^1\times \bR_+ $ & $ X_\mp $ & $ \diag(-\mu,\mu,\mu) $ & $ S^{\mp}_{\text{null}}(t) $ \\
			\hline
		\end{tabular}
	\end{table}

	\begin{definition}
		Let  $ (M,g) $ and $ (B,h) $ be a connected and orientable pseudo-Riemannian manifolds.
		A submersion $ \pi \colon M \to B $ is called a pseudo-Riemannian submersion if it satisfies the following conditions:
		\begin{enumerate}
			\item For each $ b \in B, $ the fiber $ \pi^{-1}(b) $ is a pseudo-Riemannian submanifold of $ (M,g). $
			\item There exists a horizontal distribution $ \scH $ on $ M $ such that for every $ p \in M $ the tangent space decomposes as $ T_pM = \scH_p\oplus \scV_p, $ where $ \scV \coloneqq \ker{\pi_*} $ is called the vartical distribution on $ M. $
			\item The restriction $ \pi_*|_{\scH_p} \colon (\scH_p, \pi^*h_{\pi(p)}|_{\scH_p}) \to (T_{\pi(p)}B, h_{\pi(p)}) $ is a linear isometry for an arbitrary $ p \in M $.  
		\end{enumerate}
	\end{definition}

	\begin{definition}
		A pseudo-Riemannian submersion $ \pi \colon M \to B $ is called a Killing submersion if it admits complete vartical unit Killing vector fields.
	\end{definition}

	Using timelike and spacelike Hopf coordinates is useful to discribe Killing submersions whose fibers are generated by the timelike and spacelike Killing under the action.
	\begin{enumerate}
		\item We recall that the timelike Hopf-fibration, which is described in \pref{sc:Metric of Kaluza-Klein type}, $ \bS^1\hookrightarrow \bH^3_1 \to \bH^2(4) $ along $ X_0 $ direction of the fibers is given as follows:
		\begin{equation}
			\pi_0\colon \bH^3_1 \ni (z^1,z^2) \mapsto \frac{1}{2}(|z^1|^2+|z^2|^2, -2\sqrt{-1}z^1z^2) \in \bH^2(4).
		\end{equation}
		The map $ \pi_0 $ can be seen as a pseudo-Riemannian submersion in the following context. It is easy to follow that $ X_0 $ preserves $ |z^1|^2+|z^2|^2, $ and $ z^1z^2. $ Moreover, $ \pi_0^*h_0 = \iota^*g_0 $ holds from \eqref{eq:standard-AdS3-timelike-Hopf-coordinates} and \eqref{eq:standard-H2-coordinates}. The Killing vector fields $ \{\iota_*e_\alpha = X_\alpha\}_{\alpha=0}^2 $ expressed in timelike Hopf coordinates \eqref{eq:timelike-Hopf coordinate1} and \eqref{eq:timelike-Hopf coordinate2} are given in \eqref{eq:e0}, \eqref{eq:e1} and \eqref{eq:e2}.
		Thier pushforwards under $\pi_0 $ are
		\begin{align}
			{\pi_0}_*e_0 & = 0,\\
			{\pi_0}_*e_1 & = 2\lb \frac{\cos\tau}{\sinh\chi} \frac{\partial}{\partial \varphi} - \sin\tau \frac{\partial}{\partial \chi} \rb,\\
			{\pi_0}_*e_2 & = -2 \lb \frac{\sin\tau}{\sinh\chi} \frac{\partial}{\partial \varphi} +  \cos\tau \frac{\partial}{\partial \chi}\rb. 
		\end{align}
		In other words, the decomposition of the tangent bundle $ T\bH^3_1 $ associated with the pseudo-Riemannian submersion $ \pi_0 $ determines the horizontal lifts of the orthonormal frame
		\begin{align}
			\lb 2\frac{\partial}{\partial \chi}\rb^h & =  2\frac{\partial}{\partial \chi},\\
			\lb \frac{2}{\sinh \chi}\frac{\partial}{\partial \varphi}\rb^h & =  \frac{2}{\sinh \chi}\lb \frac{\partial}{\partial \varphi} - \cosh\chi\frac{\partial}{\partial \tau} \rb.
		\end{align}
		Since the vartical subspace $ \scV_p = \vspan_\bR\{X_0|_p\} $ for each $ p\in \bH^3_1 $ of this pseudo-Riemannian submersion is generated by the unit timelike Killing, $ \pi_0 $ is a Killing submersion.
		\item The spacelike Hopf-fibration $ \bR \hookrightarrow \bH^3_1 \to \bH^2_1(4) $ along $ X_1 $ direction of the fibers is given as follows:
		\begin{equation}\label{eq:spacelike-Riemannian-submersion}
			\pi_1 \colon \bH^3_1 \ni (z^1, z^2) \mapsto \frac{1}{2}( 2\Im(\zbar^1z^2), (z^1)^2-(z^2)^2 ) \in \bH^2_1(4).
		\end{equation}
		Here, we define $ \bH^2_1(\kappa) \coloneqq \lc (t,y^1,y^2) \in \bR^3 \mid -t^2 + |y|^2 = \frac{1}{\kappa}(\kappa>0) \rc, $ which is known as the two-dimensional anti-de Sitter spacetime ($\AdS_2 $).
		The map $ \pi_1 $ can be seen as a pseudo-Riemannian submersion in the following context. It is easy to check that $ X_1 $ preserves $ \Im(\zbar^1z^2), $ and $ (z^1)^2-(z^2)^2. $ Consequently, the base space is homothetic to $ \AdS_2 $ with the metric $ \diag(-\nu,\nu) $ as seen via the pushfowards $ {\pi_1}_*X_0 $ and $ {\pi_1}_*X_2. $
		We now introduce the following spacelike Hopf coordinates:
		\begin{align}
			t^1 & = \cosh\frac{\chi}{2}\cosh\frac{\varphi}{2}\cos\frac{\tau}{2} - \sinh\frac{\chi}{2}\sinh\frac{\varphi}{2}\sin\frac{\tau}{2},\\
			t^2 & = \cosh\frac{\chi}{2}\cosh\frac{\varphi}{2}\sin\frac{\tau}{2} + \sinh\frac{\chi}{2}\sinh\frac{\varphi}{2}\cos\frac{\tau}{2},\\
			x^1 & = \cosh\frac{\chi}{2}\sinh\frac{\varphi}{2}\cos\frac{\tau}{2} - \sinh\frac{\chi}{2}\cosh\frac{\varphi}{2}\sin\frac{\tau}{2},\\
			x^2 & = \cosh\frac{\chi}{2}\sinh\frac{\varphi}{2}\sin\frac{\tau}{2} + \sinh\frac{\chi}{2}\cosh\frac{\varphi}{2}\cos\frac{\tau}{2}.
		\end{align}
		Then complex coordinates are defined by $ (z^1,z^2) = (t^1+\sqrt{-1}t^2, x^1+\sqrt{-1}x^2). $
		In these coordinates, the pseudo-Riemannian submersion $ \pi_1 $ takes the form in \eqref{eq:spacelike-Riemannian-submersion}
		\begin{equation}
			\pi_1\colon \bH^3_1 \ni (z^1,z^2) \mapsto \frac{1}{2}(\sinh\chi, \cosh\chi\cos\tau,\cosh\chi\sin\tau) \in \bH^2_1(4).
		\end{equation}
		The induced metric $ h_1 $ on $ \bH^2_1(4) $ is given by
		\begin{equation}
			h_1 = \frac{1}{4}\lb -\cosh^2\chi d\tau^2 + d\chi^2 \rb.
		\end{equation} 
		Thus, the pullback metric \begin{equation}
			 \pi_1^*h_1 = \frac{1}{4}\lc -\cosh^2\chi d\tau^2 + (d\varphi - \sinh\chi d\tau)^2 + d \chi^2 \rc
		\end{equation} 
		 coincides with the standard metric on $ \bH^3_1 $ expressed in spacelike Hopf coordinates.
		
		Next, we derive the local expressions of the Killing vector fields $ \{\iota_*e_\alpha = X_\alpha\}_{\alpha=0}^2 $ corresponding to the embedding $ \iota \colon (\tau,\varphi,\chi) \mapsto  \frac{1}{2}(\sinh\chi, \cosh\chi\cos\tau,\cosh\chi\sin\tau). $
		They are given by
		\begin{align}
			e_0 & = 2\frac{\cosh\varphi}{\cosh\chi}\lb \frac{\partial}{\partial \tau} + \sinh\chi \frac{\partial}{\partial \varphi} \rb - 2 \sinh\varphi \frac{\partial}{\partial 	\chi},\\
			e_1 & = 2\frac{\partial}{\partial \varphi},\\
			e_2 & = -2\frac{\sinh\varphi}{\cosh\chi}\lb \frac{\partial}{\partial \tau} + \sinh\chi \frac{\partial}{\partial \varphi} \rb + 2 \cosh\varphi \frac{\partial}{\partial \chi}.
		\end{align}
		Their pushfowards under $ \pi_1 $ are 
		\begin{align}
			 {\pi_1}_*e_0 & = 2\lb \frac{\cosh\varphi}{\cosh\chi} \frac{\partial}{\partial \tau} -  \sinh\varphi \frac{\partial}{\partial \chi}\rb,\\
			 {\pi_1}_*e_1 & = 0,\\
			 {\pi_1}_*e_2 & = 2\lb -\frac{\sinh\varphi}{\cosh\chi} \frac{\partial}{\partial \tau} +  \cosh\varphi \frac{\partial}{\partial \chi}\rb. 
		\end{align}
		In other words, the decomposition of the tangent bundle $ T\bH^3_1 $ associated with the pseudo-Riemannian submersion $ \pi_1 $ determines the horizontal lifts of the orthonormal frame
		\begin{align}
			\lb \frac{2}{\cosh\chi}\frac{\partial}{\partial \tau}\rb^h & =  \frac{2}{\cosh\chi} \lb \frac{\partial}{\partial \tau} + \sinh\chi\frac{\partial}{\partial \varphi} \rb,\\
			\lb 2\frac{\partial}{\partial \chi}\rb^h & = 2\frac{\partial}{\partial \chi}.
		\end{align}
		Since the vartical subspace $ \scV_p = \vspan_\bR\{X_1|_p\} $ for each $ p \in \bH^3_1 $ of this pseudo-Riemannian submersion is generated by the unit spacelike Killing, $ \pi_1 $ is a Killing submersion.
		
		\item The lightlike Hopf fibration $ \bR_+ \hookrightarrow \bH^3_1 \to \bS^1\times \bR_+ $ along $ X_+ = \frac{1}{\sqrt{2}}(X_0 + X_1) $ direction of the fiber is given as follows:
		\begin{equation}
			\pi_+ \colon \bH^3_1 \ni (z^1, z^2) \mapsto z^1-\sqrt{-1}z^2 \in \bS^1\times \bR_+.
		\end{equation}
		$ \pi_+ $ is not a pseudo-Riemannian submersion since $ \pi_+^{-1}(p) $ is not a pseudo-Riemannian submanifold for any $ p \in M $ but $ X_+ $ preserves $ z^1-\sqrt{-1}z^2 $ then $ X_+ \in \ker({\pi_+}_*). $
		Although it does not deform the metric $ g_{\lambda\mu\nu} = \diag(-\mu,\mu,\mu) $ on $ \bH^3_1 $ due to \pref{pr:invariant-metric}, there are associated homogeneous Lorentzian structures $ S^\mp_{\text{null}}(t) $ on $ (\bH^3_1,g_{\lambda\mu\nu}) $ when $ -\lambda=\mu = \nu. $
	\end{enumerate}
	
	%\section{Symmetry and the stabilizer group}

	\bibliographystyle{amsalpha}
	\bibliography{bibs}
	\nocite{coussaert1994selfdualsolutions21einstein}
	\nocite{del2024duality}
	\nocite{o1983semi}
	\nocite{blair2010riemannian}
	\nocite{welyczko2009legendre}
	\nocite{calvaruso2015killing}
	\nocite{milnor1976curvatures}
	\nocite{inoguchi2024homogeneous}
	%\nocite{cordero1997left}
%	\ \vspace{0mm} \\
	
\end{document}